\documentclass{amsart}
\usepackage{a4wide,tikz,hyperref,amsbsy,stmaryrd}

\newtheorem{lemma}{Lemma}[section]
\newtheorem{theorem}[lemma]{Theorem}
\newtheorem{claim}{Claim}
\newtheorem{corollary}[lemma]{Corollary}
\newtheorem{conjecture}[lemma]{Question}
\newtheorem{proposition}[lemma]{Proposition}
\theoremstyle{definition}
\newtheorem{definition}[lemma]{Definition}
\theoremstyle{remark}
\newtheorem{remark}[lemma]{Remark}
\newtheorem{example}[lemma]{Example}

\numberwithin{equation}{section}

\title{Recognizability for sequences of morphisms}

\author[V.~Berth\'e]{Val\'erie Berth\'e}
\author[W.~Steiner]{Wolfgang Steiner}
\address{IRIF, CNRS UMR 8243, Universit\'e Paris Diderot -- Paris 7, Case 7014, 75205 Paris Cedex 13, FRANCE}
\email{berthe@irif.fr, steiner@irif.fr}
\author[J. M. Thuswaldner]{J\"org M. Thuswaldner}
\address{Chair of Mathematics and Statistics, University of Leoben, A-8700 Leoben, AUSTRIA}
\email{joerg.thuswaldner@unileoben.ac.at}
\author[R.~Yassawi]{Reem Yassawi}
\address{ ICJ, CNRS UMR 5208, Universit\'e Claude Bernard Lyon 1, F-69622 Villeurbanne Cedex, FRANCE
}
\email{ryassawi@gmail.com}
\thanks{This work was supported by the Agence Nationale de la Recherche and the Austrian Science Fund through the projects ``Fractals and Numeration'' (ANR-12-IS01-0002, FWF I1136), ``Dyna3S'' (ANR-13-BS02-0003),  ``Discrete Mathematics'' (FWF W1230), and ``Fractals and Words: Topological, Dynamical, and Combinatorial Aspects'' (FWF P27050). It also  received funding from the European Research Council (ERC) under the European Union's
Horizon 2020 research and innovation programme,  Grant
Agreement No 648132.}
\date{\today}



\begin{document}

\begin{abstract}
We investigate different notions of recognizability for a free monoid morphism $\sigma: \mathcal{A}^* \to \mathcal{B}^*$.
Full recognizability occurs when each (aperiodic) point in $\mathcal{B}^\mathbb{Z}$ admits at most one tiling with words $\sigma(a)$, $a \in \mathcal{A}$. 
This is stronger than the classical notion of recognizability of a substitution $\sigma: \mathcal{A}^*\to\mathcal{A}^*$, where the tiling must be compatible with the language of the substitution.
We show that if  $|\mathcal A|=2$, or if  $\sigma$'s incidence matrix has  rank $|\mathcal A|$, or if $\sigma$ is permutative, then $\sigma$ is fully recognizable. 
Next we investigate the classical notion of recognizability and  improve earlier results of Moss\'{e} (1992) and Bezuglyi, Kwiatkowski and Medynets (2009), by showing that any substitution is recognizable for aperiodic points in its substitutive shift.
Finally we define  recognizability and also eventual recognizability for sequences of morphisms which define an $S$-adic shift.  We prove that a sequence of morphisms on alphabets of bounded size, such that compositions of consecutive morphisms are growing on all letters, is eventually recognizable for aperiodic points. We provide examples of eventually recognizable, but not recognizable, sequences of morphisms, and sequences of morphisms which are not eventually recognizable.  As an application, for a  recognizable sequence of morphisms, we obtain an almost everywhere bijective correspondence between the $S$-adic shift it generates,  and the measurable Bratteli-Vershik dynamical system that it defines.
\end{abstract} 
 
\maketitle

\section{Introduction}
Given a substitution $\sigma:\, \mathcal{A} \to \mathcal{A}^+$, and a long enough word~$w$ in the language generated by~$\sigma$, \emph{recognizability} is a form of injectivity of~$\sigma$ that allows one to uniquely \emph{desubstitute} most of~$w$ to another word~$v$, i.e.,  express $w$ as a concatenation of substitution words dictated by the letters in~$v$, with $v$ traditionally required to be in the substitution's language. 
Martin \cite{Martin:71} defined a notion of recognizability, and established it for a family of substitutions on a two letter alphabet.
Moss\'{e} \cite{Mosse:92, Mosse:96} showed that primitive aperiodic substitutions are recognizable, working with a combinatorial definition of recognizability; see Section~\ref{sec:defdef}  for definitions of these and other basic terms.
In \cite{Bezugly:2009}, Bezuglyi, Kwiatkowski and Medynets extended recognizability to all substitutive shifts that contain no shift-periodic points; their proof  differs essentially from Moss\'{e}'s.
Recognizability has also been widely studied in word combinatorics, under the term \emph{circularity}, in the context of D$0$L-systems, see e.g.\ \cite{Cassaigne:94,KlouSta,MigSee:93},  and  in the study of self similar tilings, where it is called the \emph{unique composition property}. 
Solomyak \cite{Solomyak:98} uses methods similar to Moss\'{e}'s to show that a  translationally finite self-affine tiling of~$\mathbb{R}^d$ has the unique composition property if and only if it is not periodic.

For substitutions, the combinatorial formulation of recognizability that Moss\'{e} uses is strongly related to the requirement above that a point be ``uniquely'' desubstitutable; see Definition~\ref{def:recog} and Remark~\ref{rem:sigma}.
Since a bi-infinite point will have infinitely many desubstitutions, uniqueness is to be understood modulo shift orbits, i.e., recognizability occurs as long as all possible desubstitutions of a point are obtained using only the shift orbit of a single point. In a dynamical context, this form of recognizability implies the presence of a natural sequence of refining (Kakutani-Rohlin)  partitions which are used extensively in key works on substitutive dynamical systems; see also
Section~\ref{sec:bratt-versh-repr}.
Before the work of Moss\'{e}, authors assumed their substitutions were recognizable. For example Host \cite{Host:1986} assumed recognizability to characterize the eigenvalues of a primitive recognizable substitutive shift,   and similarly, Anderson and Putnam  \cite{AndersonPutnam:98} assume the unique composition property  to show that the dynamics of the substitution  on the space of tilings is topologically conjugate to a shift on a stationary inverse limit.

In Definitions~\ref{def:recog} and~\ref{recognizable-morphism-sequences},  we extend the notion of recognizability to a more general setting. While we still require unique desubstitutability, we change, or extend, the set of points to which we can desubstitute.
 
We investigate establishing recognizability in the following contexts.
In the first, we ask when a single morphism $\sigma:\, \mathcal{A} \to \mathcal{B}^+$ is recognizable in~$\mathcal{A}^\mathbb{Z}$, and not only with respect to the closure of the shift orbit of one point; see Section~\ref{sec:recogn-with-resp}. 
More precisely, given a morphism and a point $y\in\mathcal{B}^\mathbb{Z}$, when does $y$ have a unique desubstitution $y = T^k \sigma(x)$ with $T$ the shift and $x \in \mathcal{A}^\mathbb{Z}$? 
A~question of this nature was recently discussed by B\'{e}daride, Hubert and Leplaideur in \cite{BedHubLep:15} (see also \cite{Emme:16})  where they study the desubstitutions of  points in a  substitutive shift by points in the full shift.
The authors were led to this question in the framework of thermodynamic formalism and Ruelle renormalization operators for potentials.

The second context, which is in fact where we began our investigations, concerns the recognizability of a sequence of free group morphisms $\boldsymbol{\sigma} = (\sigma_n)_{n\geq 0}$ with $\sigma_n:\, \mathcal{A}_{n+1} \to \mathcal{A}_n^+$. 
Such a sequence defines an \emph{$S$-adic} shift, generated by iterations of the form $\sigma_0 \circ \sigma_1 \circ \cdots \circ \sigma_n$; see Section~\ref{definition-of-S-adic-shift} for the  definition. 
If $\sigma_n = \sigma_0$ for all $n\geq 0$, then we are in the classical \emph{stationary} case of a substitutive shift. 
In fact, $\boldsymbol{\sigma}$~defines a sequence $(X_{\boldsymbol{\sigma}}^{(n)},T)_{n\geq 0}$ of shift spaces, and here, by recognizability of~$\sigma_n$, we mean that any point in $(X_{\boldsymbol{\sigma}}^{(n)},T)$ can be  desubstituted in a  unique way using~$\sigma_n$ and points in~$X_{\boldsymbol{\sigma}}^{(n+1)}$. 
We distinguish between recognizability of~$\boldsymbol{\sigma}$, where each $\sigma_n$ is recognizable, and \emph{eventual} recognizability, where all but finitely many  morphisms~$\sigma_n$ are recognizable. 
In Section~\ref{examples}, we give examples of $S$-adic shifts that are eventually recognizable but not recognizable, as well as examples of non-eventually recognizable $S$-adic shifts.

We have three main results. 
The first is Theorem~\ref{t:recognizable}, which concerns the recognizability of a single  morphism $\sigma:\, \mathcal{A} \to \mathcal{B}^+$. In particular, we show, for a large family of morphisms and for any $x \in \mathcal{A}^\mathbb{Z}$, that if $\sigma(x)$ is not shift-periodic, then $\sigma(x)$ essentially has no other desubstitutions than those involving~$x$.
In particular, if $|\mathcal{A}| = 2$, or the rank of $\sigma$'s incidence matrix is~$|\mathcal{A}|$, or $\sigma$ is \emph{rotationally conjugate} to a left (or right) {\em permutative} morphism, then our results apply. 
Here, a~morphism $\sigma:\, \mathcal{A} \to \mathcal{B}^+$ is \emph{left (right) permutative} if the first (last) letters of $\sigma(a)$ and $\sigma(b)$ are different for all distinct $a,b \in \mathcal{A}$. 
This latter condition generalizes one in \cite{BedHubLep:15}, where the authors assume that a substitution is both left and right permutative (also called \emph{marked}) to obtain a similar result. 
As the incidence matrix of any \emph{$k$-bonacci} substitution has full rank, our results include those of Emme concerning unique desubstitutions in \cite{Emme:16}.
Theorem~\ref{t:recognizable} leads to Theorem~\ref{c:rec}, where we identify a large class of sequences of morphisms that are recognizable for aperiodic points, namely sequences where each morphism is fully recognizable for aperiodic points: morphisms defined on a two-letter alphabet, morphisms with full rank, and morphisms being rotationally conjugate to a left (or right) permutative morphism.

Our second principal result is Theorem~\ref{t:evrec}. 
There we assume that we work with $S$-adic shifts where the sequence $\boldsymbol{\sigma} = (\sigma_n)_{n\geq 0}$ of morphisms is defined on alphabets of bounded size and $\boldsymbol{\sigma}$ is \emph{everywhere growing}, i.e., the length of the words $\sigma_0 \circ \sigma_1 \circ \cdots \circ \sigma_n(a)$ tends to~$\infty$ for all $a \in \mathcal{A}_{n+1}$ as $n$ tends to infinity. 
With these assumptions, we show that $\boldsymbol{\sigma}$ is eventually recognizable for aperiodic points. 
We can relax the assumption of bounded alphabets to $\liminf_{n\to\infty} |\mathcal{A}_n| < \infty$, and replace the everywhere growing condition by assuming that the points in each~$X_{\boldsymbol{\sigma}}^{(n)}$  generate a bounded number of languages. This holds in particular for minimal shifts $(X_{\boldsymbol{\sigma}}^{(n)}, T)$, as in this case all points in $X_{\boldsymbol{\sigma}}^{(n)}$ define the same language.
We also give a bound on the level~$n$ after which $\boldsymbol{\sigma}$ is recognizable, in terms of the size of the alphabets and the number of different sets of subwords.

A~slight modification of Theorem~\ref{t:evrec} gives us our third main result, Theorem~\ref{t:substrec}, where we show that each substitution is (classically) recognizable for aperiodic points in its substitutive shift. 
This extends the results of \cite{Bezugly:2009}, where aperiodicity of the whole shift was required. 

Our proofs are inspired by the result of Downarowicz and Maass in \cite{Down:2008}, where it is shown that Bratteli-Vershik topological dynamical systems are either expansive or conjugate to an odometer. 
In \cite{Bezugly:2009}, Bezuglyi, Kwiatkowski and Medynets had already harvested \cite{Down:2008} to establish recognizability for aperiodic substitutive shifts. 
Our results can thus be seen as an extension of these two works; we discuss in Section~\ref{main_S_adic-2} the challenges we have to overcome. 
We mention here that Moss\'{e}'s proof heavily depends on stationarity, and the fact that primitive substitution languages are \emph{power free}, and we have neither of these properties in our non-stationary setting. 
We also note that  Crabb, Duncan and McGregor \cite{Crabb:10} have recently revisited the (flawed) recognizability approach of Martin \cite{Martin:71}  for a two letter alphabet, and 
their proof of recognizability can be generalized to obtain recognizability of infinite $S$-adic shifts on a two letter alphabet.

We came to these question of recognizability for $S$-adic shifts for the following reason.
Under mild conditions, if $\boldsymbol{\sigma}$ is recognizable, then $(X_{\boldsymbol{\sigma}},T)$ is \emph{almost conjugate} to the \emph{natural Bratteli-Vershik} system $(X_B,\varphi)$ associated with~$\boldsymbol{\sigma}$  (Theorem~\ref{thm:Bratteli-Vershik}; see Section~\ref{sec:bratt-versh-repr} for all definitions).
The latter is not a topological dynamical system, but one that is defined on a set of full mass for any  appropriate probability measure on~$X_B$. By an almost conjugacy $\Phi$ of $(X,T)$ and $(Y,S)$, we mean that $\Phi$ is a measurable conjugacy when $(X,T)$ and $(Y,S)$ are equipped with any invariant fully supported probability measure.
The principal requirement is that our $S$-adic system has a generating sequence of partitions, and recognizability gives us this.
A~correspondence between the $S$-adic shift and the natural Bratteli-Vershik system that it defines has been established for primitive substitutive shifts by Livshits and Vershik in \cite{Livshits-Vershik}, and Canterini and Siegel in \cite{CS01b,CanSie:2001}.
In some cases, namely when the substitution is \emph{proper},  the almost conjugacy is in fact topological \cite{Durand-Host-Skau}. 
However our approach (developed in Section~\ref{sec:bratt-versh-repr}), like in \cite{Livshits-Vershik} and  \cite{CS01b,CanSie:2001},  is to drop the requirement that the Vershik map be everywhere continuous.
One exception is Theorem~\ref{tree-topological-BV-rep}, where we prove that for minimal \emph{tree shifts} $(X,T)$, the  natural Bratteli-Vershik system $(X_B,\varphi)$ associated with a \emph{return time} $S$-adic representation of $(X,T)$ is topologically conjugate to $(X,T)$.
Finally we remark that Durand and Leroy \cite{Durand-Leroy:2012} show that minimal shifts, whose complexity difference function is at most two, have $S$-adic representations which are topologically conjugate to a mild modification of the natural Bratteli-Vershik system that they define.

As applications of recognizability, let us quote estimates on the number of invariant measures,  on the rank, characterization of  spectral eigenvalues, applications to automorphism groups \cite{DonDurMassPet} and to the Sch\"utzenberger group of minimal substitutive shifts \cite{AlmCost}, and in the context of tiling spaces and related aperiodicity issues, see also  \cite{Solomyak:98,HolRadSad,FrankSadun,AndressRobinson}.

We  briefly describe the contents of the paper. 
In Section~\ref{sec:def}, we provide basic definitions and investigate the relations between the  dynamical version of recognizability and the combinatorial one.
In~Section \ref{sec:recogn-with-resp}, we provide sufficient conditions for full recognizability of a morphism.
In Sections~\ref{main_S_adic} and~\ref{main_S_adic-2} we describe our results on $S$-adic shifts, and we stress the difference between eventual recognizability and recognizability in Section~\ref{examples}.
Finally in Section~\ref{sec:bratt-versh-repr}, we discuss the implications of recognizability in terms of an $S$-adic shift's natural representation as a Bratteli-Vershik system. 

\section{Recognizability} \label{sec:def}

\subsection{Definitions of recognizability}\label{sec:defdef}
Let $\mathcal{A}$ be a finite alphabet. Then $\mathcal{A}^*$ is the free monoid of all (finite) words over $\mathcal{A}$ under the operation of concatenation, $\mathcal{A}^+$ the set of all non-empty words over $\mathcal{A}$, moreover, $\mathcal{A}^\mathbb{N}$ and $\mathcal{A}^\mathbb{Z}$ denote the one-sided and two-sided infinite sequences over~$\mathcal{A}$, respectively. Let $\mathcal{A},\, \mathcal{B}$ be finite alphabets and let $\sigma:\, \mathcal{A}^*\to \mathcal{B}^*$ be a \emph{non-erasing} morphism (also called a \emph{substitution} if $\mathcal{A} = \mathcal{B}$). By non-erasing, we mean that the image of any letter is a non-empty word. We will abuse notation and write  $\sigma:\, \mathcal{A} \to \mathcal{B}^+$.  We stress the fact that all morphisms are assumed to be non-erasing in the following.
Using concatenation, we extend $\sigma$ to~$\mathcal{A}^\mathbb{N}$ and~$\mathcal{A}^\mathbb{Z}$. 
The \emph{incidence matrix} of the morphism $\sigma$ is the $|\mathcal{B}| \times |\mathcal{A}|$ matrix~$M_\sigma=(m_{ij})$ with $m_{i,j}$ being the number of occurrences of~$i$ in $\sigma(j)$. 
Its rank is denoted by $\mathrm{rk}(M_\sigma)$; we stress the fact that $M_\sigma$ is not necessarily a square matrix.
We define the \emph{total length} of~$\sigma$ as $\interleave\sigma\interleave = \sum_{a\in\mathcal{A}} |\sigma(a)|$, where $|w|$ denotes the length of a finite word~$w$.

We use letters $x,y,z$ to denote points in two-sided shift spaces $\mathcal{A}^\mathbb{Z}$. We equip the latter with the metrizable product topology, where $\mathcal{A}$ is endowed with the discrete topology.   
The \emph{language}~$\mathcal{L}_x$ of $x = (x_n)_{n\in \mathbb Z} \in \mathcal{A}^\mathbb{Z}$ is the set of all its \emph{subwords} (or factors) $x_{[i,j)}$, $i \le j$, with $x_{[i,j)} = x_i x_{i+1}\cdots x_{j-1}$. 

Let $T: \mathcal{A}^\mathbb{Z}\rightarrow \mathcal{A}^\mathbb{Z}$ denote the (two-sided) left-shift map $(x_n)_{n\in\mathbb{Z}} \mapsto (x_{n+1})_{n\in\mathbb{Z}}$. 
Then $x \in \mathcal{A}^\mathbb{Z}$ is \emph{periodic} if $T^k(x) = x$ for some $k \ge 1$, \emph{aperiodic} otherwise.
Recall that a \emph{shift} $(X,T)$ is a dynamical system with $X$ a closed, shift-invariant subset of $\mathcal{A}^\mathbb{Z}$.
A~shift $(X,T)$ is said to be \emph{aperiodic} if each $x \in X$ is aperiodic.
A~shift $(X,T)$ is \emph{minimal} if $X$ and the empty set are the only shift-invariant closed subsets of~$X$. 
The language~$\mathcal{L}_X$ of a shift $(X,T)$ is the union of the languages of all $x \in X$. 
On the other hand, a language~$\mathcal{L}$ on~$\mathcal{A}$ which is closed under the taking of subwords defines a shift $(X_\mathcal{L}, T)$, where $X_\mathcal{L}$ is the set of points all of whose subwords belong to~$\mathcal{L}$. 
Examples of such languages are languages generated by substitutions.
Given a substitution $\sigma:\mathcal{A} \rightarrow \mathcal{A}^+$, the \emph{language}~$\mathcal{L}_{\sigma}$ defined by~$\sigma$ is
\[
\mathcal{L}_{\sigma} = \big\{w \in \mathcal{A}^*:\, \mbox{$w$ is a subword of $\sigma^n(a)$ for some  $a\in\mathcal{A}$ and $n\in \mathbb N$}\big\}.
\]
In this case we  write $X_\sigma = X_{\mathcal{L}_\sigma}$, and call $(X_\sigma, T)$ a \emph{substitutive shift}. We say that a substitution is \emph{aperiodic} if the substitutive shift that it defines is aperiodic.

We now give a ``dynamic'' definition of recognizability for morphisms.

\begin{definition}[$\sigma$-representations and  dynamic recognizability] \label{def:recog}
Let $\sigma:\, \mathcal{A} \to \mathcal{B}^+$ be a morphism and $y \in \mathcal{B}^\mathbb{Z}$.
If $y = T^k \sigma(x)$ with $x=(x_n)_{n\in \mathbb Z}
\in \mathcal{A}^\mathbb{Z}$, $k \in \mathbb{Z}$, then we say that $(k,x)$ is a \emph{$\sigma$-representation} of~$y$. 
If moreover $0 \leq k < |\sigma(x_0)|$, then $(k,x)$ is a \emph{centered $\sigma$-representation} of~$y$. 
For $X \subseteq \mathcal{A}^\mathbb{Z}$, we say that the $\sigma$-representation $(k,x)$ \emph{is in~$X$} if $x \in X$.

Given $X \subseteq \mathcal{A}^\mathbb{Z}$ and $\sigma:\, \mathcal{A} \to \mathcal{B}^+$, we say that $\sigma$ is \emph{recognizable in~$X$} if each $y \in \mathcal{B}^\mathbb{Z}$ has at most one centered $\sigma$-representation in~$X$. 
If any aperiodic point $y \in \mathcal{B}^\mathbb{Z}$ has at most one centered $\sigma$-representation in~$X$, we say that $\sigma$ is \emph{recognizable in~$X$ for aperiodic points}. 
If $\sigma$ is recognizable in~$\mathcal{A}^\mathbb{Z}$ (for aperiodic points), we say that $\sigma$ is \emph{fully recognizable (for aperiodic points)}.
\end{definition}

\begin{remark}\label{rem:sigma}
Let $\sigma:\, \mathcal{A} \to \mathcal{B}^+$, $y \in \mathcal{B}^\mathbb{Z}$, let $(X,T)$ be a shift with $X \subseteq \mathcal{A}^\mathbb{Z}$, and $Y = \bigcup_{k \in \mathbb{Z}} T^k \sigma(X)$.
\begin{enumerate}
\itemsep.5ex
\item
If $(k,x)$ is a $\sigma$-representation of~$y$, then $(k+|\sigma(x_{[0,\ell)})|, T^\ell(x))$ is a $\sigma$-representation of~$y$ for all $\ell > 0$ and $(k-|\sigma(x_{[\ell,0)})|, T^\ell(x))$ is a $\sigma$-representation of~$y$ for all $\ell < 0$.
This class of $\sigma$-representations contains a unique centered $\sigma$-representation.

\item 
The point~$y$ has a (centered) $\sigma$-representation in~$X$ if and only if $y \in Y$. 
Thus $\sigma$ is recognizable in~$X$ if and only if each $y \in Y$ has a \emph{unique} centered $\sigma$-representation in~$X$. 
The same variant holds for recognizability in~$X$ for aperiodic points. \\
Note that $Y = \bigcup_{0 \le k < \max_{a\in\mathcal{A}}|\sigma(a)|} T^k \sigma(X)$, thus $Y$ is closed and, hence, $(Y, T)$ a shift. 

\item 
Recognizability in~$X$ (for aperiodic points) is also equivalent to the fact that, for any point $x \in X$ (such that $\sigma(x)$ is aperiodic), $(0,x)$ is the unique centered $\sigma$-representation of~$\sigma(x)$.
 
\item 
Let $\sigma$ be a substitution that defines the substitutive shift $(X_\sigma,T)$. The classical definition of  recognizability of~$\sigma$ used in dynamics  corresponds to the choice $X=X_{\sigma}$ in the previous definition, i.e., $\sigma$ is recognizable in~$X_\sigma$. Recall that a substitution is \emph{primitive} if some power of its incidence matrix is positive, and it is \emph{aperiodic} if its shift contains no periodic points. 
Moss\'{e} \cite{Mosse:92,Mosse:96} showed that if $\sigma$ is primitive and aperiodic, then it is recognizable. 
She uses a combinatorial definition of recognizability that is slightly weaker than recognizability in~$X_\sigma$. (Moss\'{e}'s definition is similar to Definition~\ref{def:mossesadic} below, with the difference that Moss\'{e} works with one-sided fixed points of substitutions, whereas we consider two-sided points in~$X_\sigma$; her proof  in \cite{Mosse:92,Mosse:96} generalizes in a straightforward manner to two-sided points.) The relations between dynamically and combinatorially defined notions of recognizability are discussed in Theorem~\ref{equivalence} below, which forms a generalization of results proved in \cite{Host:1986,Queffelec:10}. 

Bezuglyi, Kwiatkowski and Medynets \cite{Bezugly:2009} extended Moss\'{e}'s  recognizability result to aperiodic substitutions. 
In Theorem~\ref{t:substrec} below, we show that any substitution~$\sigma$ is recognizable in~$X_\sigma$ for aperiodic points. \\
An example of a non-recognizable substitution with both periodic and aperiodic points in its shift is $\sigma:\, \{0,1\} \to \{0,1\}^+$, $\sigma(0) = 0010$, $\sigma(1) = 11$.
Then $X_\sigma$ contains $y = \cdots 111 \cdots$ and the fixed point~$\tilde{y}$ of $\sigma$ with $\tilde{y}_{-1} = \tilde{y}_0 = 0$, which is aperiodic; $y$~has two centered $\sigma$-representations, while $\tilde{y}$ has a unique centered $\sigma$-representation.

\item
Let $\bar{X}_\sigma$ be the set of all one-sided points $\bar{x} = (\bar{x}_n)_{n\in \mathbb{N}}$ which have a continuation on the left to a point in~$X_\sigma$, i.e., points $\bar{x}$ for which there exists  a point $x=(x_n)_{n\in \mathbb{Z}}\in X_\sigma$ such that $x_n=\bar{x}_n$ for all $n\in\mathbb N$. Then $\sigma$ is \emph{unilaterally recognizable} if each point $\bar{x}\in \bar{X}_\sigma$  has a unique centered $\sigma$-representation in~$\bar{X}_\sigma$.  In Moss\'{e}'s earlier work \cite{Mosse:92}, once again working with a combinatorial definition of recognizability, she  showed that a primitive aperiodic substitution $\sigma$ is  not necessarily unilaterally recognizable. Indeed, unicity can fail at the start of the word $\bar{x}$. This entails that additional conditions are required for unilateral recognizability.
In this context, the notion of \emph{postfixed} substitutions is of interest: a~substitution is said to be \emph{postfixed} if for any letters $a$ and~$b$, if $\sigma(a) \neq \sigma( b)$, then $\sigma(a)$ is not a suffix of~$\sigma(b)$; see~\cite{Host:1986}. 
While postfixed substitutions are unilaterally recognizable, this is not necessarily true when the image of a letter is a strict prefix of the image of another letter. 
As an illustration of this fact, consider the substitution $\sigma:\, 0 \mapsto 010, \, 1 \mapsto 10$, $x = \sigma(0) \sigma^{2}(0) \cdots$. Note that $x=T^2(u)$ where $u$ is the one-sided fixed point of $\sigma$ starting with 1.
Then we have $\sigma(0x) = 0 \sigma(1x)$, where both $0x$ and $1x$ belong to the one-sided shift defined by~$\sigma$ as $\sigma^n(1) = 1 0 \sigma(0) \cdots \sigma^{n-1}(0)$ and $\sigma^n(01) = \sigma^{n-1}(01) \sigma^n(0) = 01 \sigma(0) \cdots \sigma^n(0)$ for all $n \ge 2$; cf.\ \cite[Example~3.2]{Crabb:10}. 
\end{enumerate}
\end{remark}

Next we give our variant of Moss\'e's combinatorial definition of recognizability. To this end we need the following notion of $\sigma$-cutting points.

\begin{definition}[Cutting points]
Let $(k,x)$ be a $\sigma$-representation for some morphism~$\sigma$, and let $\ell \in \mathbb{Z}$.
The \emph{$\ell$-th $\sigma$-cutting point} of $(k,x)$ is $|\sigma(x_{[0,\ell)})|-k$ if $\ell \ge 0$, $-|\sigma(x_{[\ell,0)})|-k$ if $\ell < 0$. 
Denote by $C_\sigma(k,x)$ the collection of $\sigma$-cutting points of $(k,x)$, i.e., 
\[
C_\sigma(k,x) = \{|\sigma(x_{[0,\ell)})|-k:\, \ell \ge 0\} \cup \{-|\sigma(x_{[\ell,0)})|-k:\, \ell < 0\};
\]
see Figure \ref{Figure_1}.
Two $\sigma$-representations $(k,x), (k',x')$ have a \emph{common $\sigma$-cut} if $C_\sigma(k,x) \cap C_\sigma(k',x') \ne \emptyset$. 
\end{definition}

\begin{figure}[ht]\begin{tikzpicture}

\draw(-6.4,0)--(7.2,0) (-5.5,-.4)node[below]{$(-2)$-nd}--(-5.5,.2) (-2.5,-.4)node[below]{$(-1)$-st}--(-2.5,.2) (0,-.4)node[below]{$0$-th}--(0,.2) (1.3,-.4)node[below]{$1$-st}--(1.3,.2) (3.4,-.4)node[below]{$2$-nd}--(3.4,.2);
\node[above] at (-5.9,0){$\cdots\vphantom{y_|}$};
\node[below] at (-5.9,0){$\cdots\vphantom{(}$};
\node[above] at (-4,0){$y_{-|\sigma(x_{-2}x_{-1})|-k}\cdots$};
\node[below] at (-4,0){$\sigma(x_{-2})$};
\node[above] at (-1.25,0){$y_{-|\sigma(x_{-1})|-k}\cdots$};
\node[below] at (-1.25,0){$\sigma(x_{-1})$};
\node[above] at (.65,0){$y_{-k}\cdots\vphantom{y_|}$};
\node[below] at (.65,0){$\sigma(x_0)$};
\node[above] at (2.35,0){$y_{|\sigma(x_0)|-k}\cdots$};
\node[below] at (2.35,0){$\sigma(x_1)$};
\node[above] at (4.6,0){$y_{|\sigma(x_0x_1)|-k}\cdots$};
\node[below] at (4.6,0){$\sigma(x_2)\quad\cdots$};
\node[below] at (6,-.4){$\sigma$-cutting point};
\end{tikzpicture}
\caption{A~$\sigma$-representation $(k,x)$ of~$y$ and its $\sigma$-cutting points; if $(k,x)$ is centered, then $y_0$ is a letter of $\sigma(x_0)$.}
\label{Figure_1}
\end{figure}
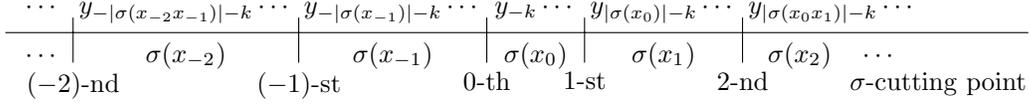

\begin{definition}[Moss\'e's recognizability] \label{def:mossesadic}  
Let $\sigma:\, \mathcal{A} \to \mathcal{B}^+$ be a morphism, and let $x \in \mathcal{A}^\mathbb{Z}$, $y = \sigma(x)$. 
We say that $\sigma$ is \emph{recognizable in the sense of Moss\'{e} for~$x$} if there exists $\ell$ such that, for each $m \in C_\sigma(0,x)$, $m' \in \mathbb{Z}$, $y_{[m-\ell,m+\ell)} =  y_{[m'-\ell,m'+\ell)}$ implies that $m' \in C_\sigma(0,x)$.
\end{definition}

Note that the constant of recognizability~$\ell$ is computable for primitive morphisms \cite{Durand-Leroy}.
For a discussion of the various notions of recognizability, including Martin's one~\cite{Martin:71}, see~\cite{Mosse:96,Crabb:10}. 

\subsection{Relations between different definitions of recognizability} \label{sec:recdef}

We now show that the two definitions of recognizability given in Section~\ref{sec:defdef} are closely related. In particular, we prove the following result. Given a point $x$, the \emph{shift generated by $x$} is the shift defined by $\mathcal L_x$.

\begin{theorem}\label{equivalence}
Let $\sigma:\, \mathcal{A} \to \mathcal{B}^+$ be a morphism, $x \in \mathcal{A}^\mathbb{Z}$, and $(X,T)$ be the shift generated by~$x$. 
Then the following assertions hold. 
\begin{enumerate}
\item
If $\sigma$ is recognizable in~$X$, then $\sigma$ is recognizable in the sense of Moss\'{e} for~$x$.
\item 
If $(X,T)$ is minimal, $\sigma$ is injective on letters, and $\sigma$ is recognizable in the sense of Moss\'{e} for~$x$, then $\sigma$ is recognizable in~$X$. 
\end{enumerate}
\end{theorem}

\begin{proof}
This proof is an appropriate generalization of arguments in \cite{Host:1986}.
Let $y = \sigma(x)$ and $Y = \bigcup_{k\in\mathbb{Z}} T^k \sigma(X)$.

Suppose that $\sigma$ is recognizable in~$X$. 
We claim that $\sigma$ is recognizable in the sense of Moss\'{e} for~$x$. 
For if not, then for all~$i$, there are integers $m_i$ and $m'_i$ such that $m_i \in C_\sigma(0,x)$, $m'_i \notin C_\sigma(0,x)$, and $d(T^{m_i}(y),  T^{m'_i}(y)) \to 0$ as $i \to \infty$. 
Passing to a subsequence, we can suppose that $T^{m_i}(y) \to \tilde{y} \in \sigma(X)$ and that there are some fixed $a \in \mathcal{A}$, $0 < k < |\sigma(a)|$, such that $T^{m'_i}(y) \in T^k \sigma([a])$ for all~$i$, where $[a]$ denotes the cylinder $\{\tilde{x} \in X:\, \tilde{x}_0 = a\}$. 
Hence $\tilde{y} \in \sigma(X) \cap T^k \sigma([a])$, and this contradicts the assumption that $\sigma$ is recognizable in~$X$.

Conversely, suppose that $(X,T)$ is minimal, $\sigma$ is injective on letters, and $\sigma$ is recognizable in the sense of Moss\'{e} for~$x$.

\begin{claim}\label{claim1}
For $\tilde{y} \in Y$, we have $\tilde{y} \in \sigma(X)$ if and only if whenever for a sequence $(m_i)$ we have $\lim_{i\to\infty} T^{m_i}(y) = \tilde{y}$, then $m_i \in C_\sigma(0,x)$ for all $i$ large enough.
\end{claim}

Let $\tilde{y} \in Y = \bigcup_{k\in\mathbb{Z}} T^k \sigma(X)$.
Minimality of $(X,T)$ implies the existence of a sequence $(m_i)$ satisfying $\lim_{i\to\infty} T^{m_i}(y) = \tilde{y}$. 
Suppose first that there is $i_0 \in \mathbb{N}$ such that $m_i \in C_\sigma(0,x)$ for all $i \ge i_0$, and let $k_i$ be such that $T^{m_i}(y) = \sigma(T^{k_i}(x))$. 
Moving to a convergent subsequence if necessary, we have $\lim_{i\to\infty} T^{k_i}(x) = \tilde{x} \in X$, so that $\tilde{y} \in \sigma(X)$.
Conversely, suppose that $\tilde{y} = \sigma(\tilde{x})$ with $\tilde{x} \in Y$. 
Let $(m_i)$ be an arbitrary sequence satisfying $T^{m_i}(y) \to \tilde{y}$. 
Minimality implies that there is a  sequence $(k_i)$ with $T^{k_i}(x) \to \tilde{x}$, so $\sigma(T^{k_i}(x)) \to \tilde{y}$. 
Fix $\ell$ as in Definition~\ref{def:mossesadic}, and let the sequence $(h_i)$ satisfy $\sigma(T^{k_i}(x)) = T^{h_i}(y)$. 
Then each $h_i \in C_\sigma(0,x)$, and $y_{[m_i-\ell, m_i +\ell ]} = y_{[h_i-\ell, h_i +\ell ]}$ for large~$i$, so that $m_i \in C_\sigma(0,x)$, and we have proved Claim~\ref{claim1}.
 
\begin{claim} \label{claim2}
The set $\sigma(X)$ is a clopen subset of~$Y$.
\end{claim}

The set $\sigma(X)$ is closed in~$Y$ by compactness and continuity of~$\sigma$, and we see that its complement is closed using Claim~\ref{claim1}: indeed, by a Cantor diagonal argument, we see that the set of all $\tilde{y} \in Y$ satisfying $\lim_{i\to\infty} T^{m_i}(y) = \tilde{y}$ for a sequence $(m_i)$ with $m_i \notin C_\sigma(0,x)$ for arbitrarily large~$i$ is closed. 
This proves Claim~\ref{claim2}.

\begin{claim}\label{claim3}
Let $\tilde{x} \in X$. 
Then $T^m(\sigma(\tilde{x})) \in \sigma(X)$ if and only if $m \in C_\sigma(0,\tilde{x})$.  
\end{claim}
 
Set $\tilde{y} = \sigma(\tilde{x})$. 
If $m \in C_\sigma(0,\tilde{x})$, then $T^m(\tilde{y}) = \sigma(T^k(\tilde{x}))$ for some $k \in \mathbb{Z}$ and thus $T^m(\tilde{y}) \in \sigma(X)$.
For the converse, write $\tilde{x} = \lim_{i\to\infty} T^{k_i}(x)$ and $m_i = |\sigma(x_{[0,k_i)})|$, hence $\tilde{y} = \lim_{i\to\infty} T^{m_i}(y)$ and $\lim_{i\to\infty} T^{m+m_i}(y) = T^m(\tilde{y}) \in \sigma(X)$.
Here we have assumed that the $k_i$'s are positive; the other case is similar. 
By Claim~\ref{claim1}, we have $m + m_i \in C_\sigma(0,x)$ for all large~$i$. 
As $\lim_{i\to\infty} T^{k_i}(x) = \tilde{x}$, this implies that $m \in C_\sigma(0,\tilde{x})$. 
 
\begin{claim}\label{claim4}
The map $\sigma:\, X \to \sigma(X)$ is a homeomorphism.
\end{claim}

This is where we use that $\sigma$ is injective on letters. 
By the compactness of~$X$, it is sufficient to show that $\sigma:\, X \to \sigma(X)$ is injective. 
If $\sigma(\tilde{x}) = \sigma(x')$, then by Claim~\ref{claim3}, the first return time of $\sigma(\tilde{x})$ and $\sigma(x')$ under the shift to $\sigma(X)$ is $|\sigma(\tilde{x}_0)| = |\sigma(x'_0)|$. 
Thus $\sigma(\tilde{x}_0) = \sigma(x'_0)$ and injectivity  on letters of $\sigma$ now implies that $\tilde{x}_0 = x'_0$. Iterating this argument yields $x=x'$.

\begin{claim}\label{claim5}
The collection 
\begin{equation}\label{eq:parti}
\mathcal{P} = \{ T^k \sigma([a]):\, a \in \mathcal{A},\, 0\leq k<|\sigma(a)| \}
\end{equation}
is a clopen partition of~$Y$, so that $\sigma$ is recognizable in~$X$.
\end{claim}

It is clear that $\mathcal{P}$ is a cover of~$Y$, and the sets in~$\mathcal{P}$ are clopen by Claims~\ref{claim2} and~\ref{claim4}.  
Suppose that $\tilde{y} \in T^k \sigma([a]) \cap T^j \sigma([b])$, 
$0 \le k<|\sigma(a)|$, $0 \le j<|\sigma(b)|$, and $k\geq j$.
By shifting if necessary, we may assume that $j=0$, so that $\tilde{y} \in \sigma(X)$. 
Since $T^k \sigma(\tilde{x}) = \tilde{y} \in \sigma(X)$ for some $\tilde{x} \in [a]$, we have $k = 0$ by Claim~\ref{claim3}. 
Hence $j = k = 0$ and, with Claim~\ref{claim4}, we conclude that $\mathcal{P}$ is a partition. 
\end{proof}

Thus, while the substitution given in (5), Remark  \ref{rem:sigma}, is not unilaterally recognizable, it is, by Moss\'{e}'s Theorem,  recognizable in the sense of Moss\'{e}. Hence by Theorem \ref{equivalence}, $\sigma$ is recognizable in~$X_\sigma$.

We remark that for primitive substitutions, we can drop the assumption of injectivity in assertion~(2) of Theorem~\ref{equivalence}, see \cite{Mosse:92}.

\section{Recognizability for aperiodic points in the full shift} \label{sec:recogn-with-resp}

\subsection{Statement of the result}
In this section we show that a morphism is fully recognizable for aperiodic points under mild conditions; see Theorem \ref{t:recognizable}. 
This includes morphisms $\sigma:\, \mathcal{A} \rightarrow \mathcal{B}^+$ whose incidence matrices have rank equal to~$|\mathcal{A}|$, and thus in particular includes any substitution with an invertible incidence matrix. 
Hence, \emph{irreducible Pisot} substitutions~$\sigma$, i.e., those where the minimal polynomial of~$M_\sigma$ is the minimal polynomial of a Pisot number, are fully recognizable.
It also includes any morphism on two letters, any left (or right) \emph{permutative} morphism, and any morphism that is \emph{rotationally conjugate} to such a morphism. 
Here, a~morphism $\sigma:\, \mathcal{A} \to \mathcal{B}^+$ is \emph{left permutative} if the first letters of $\sigma(a)$ and $\sigma(b)$ are different for all distinct $a,b \in \mathcal{A}$.
It is \emph{right permutative} if the last letters of $\sigma(a)$ and $\sigma(b)$ are different for all distinct $a,b \in \mathcal{A}$. 
Two morphisms $\sigma, \tilde{\sigma}:\, \mathcal{A} \to \mathcal{B}^+$ are \emph{rotationally conjugate} if there is a word $w \in \mathcal{B}^*$ such that $\sigma(a) w = w \tilde{\sigma}(a)$ for all $a \in \mathcal{A}$ or $w \sigma(a) = \tilde{\sigma}(a) w$ for all $a \in \mathcal{A}$.

\begin{theorem} \label{t:recognizable}
Let $\sigma:\, \mathcal{A} \to \mathcal{B}^+$ be a morphism such that 
\begin{itemize}
\item $\mathrm{rk}(M_\sigma) = |\mathcal{A}|$,  or
\item $|\mathcal{A}| = 2$, or
\item  $\sigma$ is (rotationally conjugate to) a left or right permutative morphism.
\end{itemize}
Then $\sigma$ is fully recognizable for aperiodic points.
\end{theorem}

\subsection{Permutative morphisms and rotational conjugacy}
We first consider the theorem for left or right permutative morphisms. 
Then we reduce the other cases to this one. 
In particular, Lemma~\ref{l:permutative} below states  that  left or right permutative morphisms are fully recognizable for aperiodic points. 
We start with the following uniqueness result of $\sigma$-representations along orbits.

\begin{lemma}\label{lem:uniqueorbit}
Let $\sigma:\, \mathcal{A} \to \mathcal{B}^+$ be a morphism.
Let $y \in \mathcal{B}^{\mathbb{Z}}$ be an aperiodic point, and let  $(k,x)$, $(k',x')$ be centered $\sigma$-representations of~$y$ with $x' = T^\ell (x)$ for some $\ell \in \mathbb{Z}$. 
Then $(k,x) = (k',x')$. 
In other words, aperiodic points have a unique centered representation along an orbit.     
\end{lemma}

\begin{proof}
Suppose that $(k,x)\neq(k',x')$. 
Since $y$ is aperiodic this implies $\ell\neq 0$ and we have
\[
y = T^k \sigma(x) = T^{k'} \sigma(T^\ell(x)) = 
\begin{cases}T^{k'+|\sigma(x_{[0,\ell)})|} \sigma(x) & \text{if}\ \ell>0, \\ T^{k'-|\sigma(x_{[\ell,0)})|} \sigma(x) & \text{if}\ \ell<0.\end{cases}
\]
We have assumed that $0\leq k<|\sigma(x_0)|$ and $0\leq k' < |\sigma(x'_0)| = |\sigma(x_\ell)|$, thus 
\begin{align*}
|\sigma(x_0)| > k = k'+|\sigma(x_{[0,\ell)})| \geq |\sigma(x_0)| & \quad \text{if}\ \ell>0, \\
0 \leq k = k'-|\sigma(x_{[\ell,0)})| < 0 & \quad\text{if}\ \ell<0, 
\end{align*} 
and either inequality leads to a contradiction. 
\end{proof}

\begin{lemma} \label{l:permutative}
Let $\sigma:\, \mathcal{A} \to \mathcal{B}^+$ be a left or right permutative morphism, and let $y \in \mathcal{B}^\mathbb{Z}$. 
If $y$ has two centered $\sigma$-representations, then $y$ is periodic. 
\end{lemma}

\begin{proof}
Let $(k,x), (k',x')$ be distinct centered $\sigma$-representations of~$y$. 
Assume that $\sigma$ is left permutative, the right permutative case being symmetric.
Then $\sigma$ is injective on $\mathcal{A}^\mathbb{N}$, i.e., on right infinite one-sided sequences.

Suppose that $C_\sigma(k,x) \cap C_\sigma(k',x') \ne \emptyset$. 
Let $ h$ be in the intersection, and let $ \ell,\ell'$ be such that $ h$ is the $\ell$-th $\sigma$-cutting point of $(k,x)$, as well as the $\ell'$-th $\sigma$-cutting point of $(k',x')$.
By the injectivity of $\sigma$ on~$\mathcal{A}^\mathbb{N}$, we have $x_{[ \ell,\infty)} = x'_{[ \ell',\infty)}$, and all cutting points from $ h$ onwards are common. 
If all $\sigma$-cutting points are common then we are done, since this implies that $x = T^{\ell- \ell'}(x')$ and hence $y$ is periodic by Lemma~\ref{lem:uniqueorbit}. Otherwise let $H$ be the smallest element of $C_\sigma(k,x) \cap C_\sigma(k',x')$, and 
let $L,L'$  be such that $H$ is  the $L$-th $\sigma$-cutting point of $(k,x)$ and the $L'$-th $\sigma$-cutting point of $(k',x')$.
Finally if $C_\sigma(k,x) \cap C_\sigma(k',x') = \emptyset$, then we set $H = L = L' = \infty$. 

Let now $h   \in C_\sigma(k,x) \backslash C_\sigma(k',x') $ be the $\ell$-th $\sigma$-cutting point of $(k,x)$, and let $\ell'$ be such that $h$ lies (strictly) between the $\ell'$-th and $(\ell'{+}1)$-st $\sigma$-cutting points of $(k',x')$.
Let $h'$ be the $\ell'$-th $\sigma$-cutting point of $(k',x')$, see Figure~\ref{f:permutative}.
Then the first letter of $\sigma(x_\ell)$ is the $(h{-}h'{+}1)$-st letter of $\sigma(x'_{\ell'})$, and left permutativity implies that $x_\ell$ is determined by $x'_{\ell'}$ and the difference $h-h'$. 
Inductively, we obtain that for each~$\ell<L$, $x_{[\ell,L)}$ and $x'_{[\ell',L')}$ are determined by $x'_{\ell'}$ and $h-h'$.
Since there are only finitely many possibilities for $x'_{\ell'}$ and $h-h'$, we have an infinite set of indices $\ell < L$ having all the same $x'_{\ell'}$ and $h-h'$. 
At all these~$\ell$'s, we get the same word $x_{[\ell,\infty)}$, which is periodic of the form  $x_{[\ell,\infty)} = x_{[\ell_0,\ell_1)}  x_{[\ell_0,\ell_1)} \ldots$ whenever $\ell_0, \ell_1$ are arbitrary indices in this set. 
Therefore $x$, and thus~$y$, are periodic.
\end{proof}

\begin{figure}[ht]
\begin{tikzpicture}
\draw(-1.5,0)--(3.5,0) (-1.5,.65)--(3.5,.65) (0,.45)--(0,1.05) (-1.2,.2)--(-1.2,-.4) (1.1,.2)--(1.1,-.4);
\node[above] at (1.1,.65){$\sigma(x_\ell) \quad \cdots$};
\node[above] at (-.6,0){$y_{h'}\cdots\vphantom{y_{|x'_{\ell'}}}$};
\node[above] at (.55,0){$y_h\cdots\vphantom{y_{|x'_{\ell'}}}$};
\node[above] at (2.25,0){$y_{h'+|\sigma(x'_{\ell'})|}\cdots$};
\node[below] at (-.05,0){$\sigma(x'_{\ell'})$};
\node[below] at (2.25,0){$\cdots\vphantom{(x'_{\ell'}}$};
\end{tikzpicture}
\caption{In the proof of Lemma~\ref{l:permutative}, $x_\ell$ is determined by $y_h$ since $\sigma$ is left permutative, thus $x_\ell$ is determined by $x'_{\ell'}$ and $h-h'$.} \label{f:permutative}
\end{figure}
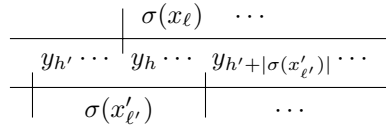
 
For the following lemma, we use that rotational conjugation only shifts the $\sigma$-cutting points. 

\begin{lemma} \label{l:conjugaterecognizable}
Let $\sigma, \tilde{\sigma}:\, \mathcal{A} \to \mathcal{B}^+$ be rotationally conjugate morphisms.
Then, for each $y \in \mathcal{B}^\mathbb{Z}$, the number of centered $\sigma$-representations of~$y$ is equal to the number of centered $\tilde{\sigma}$-representations~of~$y$.
\end{lemma}

\begin{proof}
Assume w.l.o.g.\ that $w \sigma(a) = \tilde{\sigma}(a) w$ for all $a \in \mathcal{A}$, and let $y \in \mathcal{B}^\mathbb{Z}$. 
Then $(k,x)$ is a $\sigma$-representation of~$y$ if and only if $(k+|w|,x)$ is a $\tilde{\sigma}$-representation of~$y$.
\end{proof}

If $\mathcal{A} = \{a,b\}$ and $|\sigma(\mathcal{A}^\mathbb{Z})| \ge 2$, then $\sigma$ is rotationally conjugate to a left permutative morphism~$\tilde{\sigma}$. 
Indeed, if $\tilde{\sigma}(a)$ and $\tilde{\sigma}(b)$ start with the same letter for all morphisms $\tilde{\sigma}$ that are conjugate to~$\sigma$, then we have $\sigma(aa\cdots) = \sigma(bb\cdots)$, thus $\sigma(a)$ and $\sigma(b)$ are powers of the same word by the Fine-Wilf theorem \cite{Fine:65} and, hence, $|\sigma(\mathcal{A}^\mathbb{Z})| = 1$.
If $\sigma(\mathcal{A}^\mathbb{Z})$ contains an aperiodic point, then it is an infinite set.
Together with Lemmas~\ref{l:permutative} and~\ref{l:conjugaterecognizable}, this implies that $\sigma$ is fully recognizable for aperiodic points when $|\mathcal{A}| = 2$.
We give an alternative proof in Section~\ref{sec:proof-theor-reft:r}.

\subsection{Composing morphisms}
We now relate the recognizability properties of the composition of two morphisms $\tau \circ \sigma$ to those of the single morphisms $\tau$ and $\sigma$. 
Here, we write $\tau \sigma$ for $\tau \circ \sigma$. 

\begin{lemma}\label{lem:tele}
Let $\sigma:\, \mathcal{A} \to \mathcal{B}^+$ and $\tau:\, \mathcal{B}\to\mathcal{C}^+$ be morphisms, $(X,T)$ a shift with $X \subseteq \mathcal{A}^\mathbb{Z}$, and $Y = \bigcup_{k\in\mathbb{Z}} T^k \sigma(X)$. 
Then $\tau \sigma$ is recognizable in~$X$ if and only if $\sigma$ is recognizable in~$X$ and $\tau$ is recognizable in~$Y$. 
If $\sigma$ is recognizable in~$X$ for aperiodic points and $\tau$ is recognizable in~$Y$ for aperiodic points, then $\tau \sigma$ is recognizable in~$X$ for aperiodic points.
If $\tau \sigma$ is recognizable in~$X$ for aperiodic points, then $\tau$ is recognizable in~$Y$ for aperiodic points. 
\end{lemma}

\begin{proof}
We first establish a bijection between centered $\tau \sigma$-representations $(m,x)$ and pairs of a centered $\sigma$-representation $(k,x)$ and a centered $\tau$-representation $(\ell,T^k \sigma(x))$.
Given $(m,x)$, let $z = T^m \tau \sigma(x)$. 
Then $z_{[-m,0)}$ is a prefix of $\tau \sigma(x_0)$ and there are unique $p \in \mathcal{B}^*$,  $a \in \mathcal{B}$, such that $pa$ is a prefix of $\sigma(x_0)$ and $|\tau(p)| \le m < |\tau(pa)|$.
Let $k = |p|$, $\ell = m-|\tau(p)|$, and $y = T^k\sigma(x)$. 
Then $0 \le k < |\sigma(x_0)|$ and $0 \le \ell < |\tau(a)| = |\tau(y_0)|$, thus $(k,x)$ is a centered $\sigma$-representation and $(\ell,y)$ is a centered $\tau$-representation.
On the other hand, given $(k,x)$ and $(\ell,y)$ with $y = T^k \sigma(x)$, then we find the centered $\tau \sigma$-representation $(m,x)$ from above by setting $m = |\tau(y_{[-k,0)})| + \ell$.
Note that $T^m \tau \sigma(x) = T^\ell \tau(T^k\sigma(x))$.

Suppose that $z \in \mathcal{C}^\mathbb{Z}$ has two centered $\tau \sigma$-representations $(m,x) \ne (m',x')$ in~$X$. 
Let $(k,x)$, $(\ell,T^k \sigma(x))$, $(k',x')$ and $(\ell',T^{k'} \sigma(x'))$ be the associated centered $\sigma$- and $\tau$-representations.
Then $(\ell,T^k \sigma(x))$ and $(\ell',T^{k'} \sigma(x'))$ are centered $\tau$-representations of~$z$ in~$Y$.
If they are equal, then we have $(k,x) \ne (k',x')$ and $y:=T^k \sigma(x) = T^{k'} \sigma(x')$; thus $y$  has two centered $\sigma$-representations in~$X$.
Hence, if $\tau \sigma$ is not recognizable in~$X$, then $\sigma$ is not recognizable in~$X$ or $\tau$ is not recognizable in~$Y$.
As $y$ is aperiodic when $z$ is aperiodic, this also holds with recognizability for aperiodic~points.

Next suppose that $y \in \mathcal{B}^\mathbb{Z}$ has two centered $\sigma$-representations $(k,x) \ne (k',x')$ in~$X$. 
Let $m = |\tau(y_{[-k,0)})|$ and $m' = |\tau(y_{[-k',0)})|$. 
Then $(m,x) \ne (m',x')$ are two centered $\tau \sigma$-representations of $\tau(y)$. 
Hence, if $\sigma$ is not recognizable in~$X$, then $\tau \sigma$ is not recognizable in~$X$. 
(We cannot conclude that this holds for aperiodic points since $\tau(y)$ can be periodic when $y$ is aperiodic.)

Finally, suppose that $z \in \mathcal{C}^\mathbb{Z}$ has two centered $\tau$-representations $(\ell,y) \ne (\ell',y')$ in~$Y$. 
Let $(k,x)$ and $(k',x')$ be centered $\sigma$-representations in~$X$ of $y$ and~$y'$, respectively.
Let $m = |\tau(y_{[-k,0)})|+\ell$ and $m' = |\tau(y'_{[-k',0)})|+\ell'$. 
Then $(m,x)$ and $(m',x')$ are centered $\tau \sigma$-representations of~$z$.
We have $(m,x) \ne (m,x')$ since $x = x'$ implies that $m \ne m'$ by the first paragraph of the proof.
Therefore, non-recognizability of~$\tau$ in~$Y$ implies non-recognizability of $\tau \sigma$ in~$X$.
This relation also holds with recognizability for aperiodic points. 
\end{proof}

\subsection{Injective and non-injective morphisms}
The following two lemmas are inspired by~\cite[Case~(1) of the proof of Theorem~1]{Down:2008}.
They are combinatorial interpretations of one of their arguments, where the existence of a common cut means that the morphism is not injective on two-sided sequences, and thus not injective on right or on left infinite sequences. 
If the image of one letter is a concatenation of images of other letters, then we can remove the letter from the alphabet. 
Otherwise, we can recode by a morphism, on the same alphabet, where the image of one letter is shorter than in the original morphism and the other images do not change.
In both cases, the total length of the morphism decreases. 

\begin{lemma} \label{l:reducetotallength}
Let $\sigma:\, \mathcal{A} \to \mathcal{B}^+$ be a morphism that is not injective on~$\mathcal{A}^\mathbb{N}$.
Then we have $\sigma = \tilde{\sigma} \tau$ with morphisms $\tau:\, \mathcal{A} \to \tilde{\mathcal{A}}^+$, $\tilde{\sigma}:\, \tilde{\mathcal{A}} \to \mathcal{B}^+$, such that
\begin{enumerate}
\renewcommand{\theenumi}{\roman{enumi}}
\item
 $|\tilde{\mathcal{A}}| < |\mathcal{A}|$, or $|\tilde{\mathcal{A}}| = |\mathcal{A}|$ and $\tau$ is injective on~$\mathcal{A}^\mathbb{N}$, 
\item
$\interleave\tilde{\sigma}\interleave < \interleave\sigma\interleave$, and 
\item 
each $\tilde{\sigma}(a)$, $a \in \tilde{\mathcal{A}}$, is a prefix of some $\sigma(b)$, $b \in \mathcal{A}$. 
\end{enumerate}
\end{lemma}

\begin{proof}
To prove the claim, choose $x, x' \in \mathcal{A}^\mathbb{N}$ with $\sigma(x) = \sigma(x')$ and $x_k \ne x'_k$ for some $k \ge 0$.
Assume that $x_{[0,k)} = x'_{[0,k)}$ and, w.l.o.g., $|\sigma(x'_k)| > |\sigma(x_k)|$.
Then we have $\sigma(x'_k) = \sigma(x_{[k,\ell)})\, v$ for some $\ell > k$ and some prefix $v$ of $\sigma(x_\ell)$ with $v \ne \sigma(x_\ell)$. 
If $v$ is empty, then the statement of the lemma holds with $\tilde{\mathcal{A}} = \mathcal{A} \setminus \{x'_k\}$, $\tilde{\sigma}$~being the restriction of $\sigma$ to $\tilde{\mathcal{A}}$, and $\tau$ defined by $\tau(x'_k) = x_{[k,\ell)}$, $\tau(a) = a$ otherwise. 
If $v$ is non-empty, then we can take $\tilde{\mathcal{A}} = \mathcal{A}$, $\tilde{\sigma}(x'_k) = v$, $\tilde{\sigma}(a) = \sigma(a)$ otherwise, $\tau(x'_k) = x_{[k,\ell)} x'_k$, $\tau(a) = a$ otherwise; in this case, $\tau$ is injective on~$\mathcal{A}^\mathbb{N}$.
\end{proof}

For the proof of Theorem~\ref{t:recognizable} and in Section~\ref{main_S_adic-2}, we need the following consequences of Lemma~\ref{l:reducetotallength}.

\begin{lemma} \label{l:reducealphabet}
Let $\sigma:\, \mathcal{A} \to \mathcal{B}^+$ be a morphism that is not injective on~$\mathcal{A}^\mathbb{Z}$.
Then we have $\sigma = \tilde{\sigma} \tau$ with morphisms $\tau:\, \mathcal{A} \to \tilde{\mathcal{A}}^+$, $\tilde{\sigma}:\, \tilde{\mathcal{A}} \to \mathcal{B}^+$, such that $|\tilde{\mathcal{A}}| < |\mathcal{A}|$ and $\interleave\tilde{\sigma}\interleave < \interleave\sigma\interleave$.
If $\sigma$ is not injective on~$\mathcal{A}^\mathbb{N}$, then we can choose $\tilde{\sigma}$ such that each $\tilde{\sigma}(a)$, $a \in \tilde{\mathcal{A}}$, is a prefix of some $\sigma(b)$, $b \in \mathcal{A}$. 
\end{lemma}

\begin{proof}
If $\sigma$ is not injective on~$\mathcal{A}^\mathbb{Z}$, then it is not injective on~$\mathcal{A}^\mathbb{N}$ or on~$\mathcal{A}^{-\mathbb{N}}$. 
Assume that $\sigma$ is not injective on~$\mathcal{A}^\mathbb{N}$, the other case being symmetric.
Let $\sigma = \sigma_1 \tau_1$, $\tau_1:\, \mathcal{A} \to \mathcal{A}_1^+$, $\sigma_1:\, \mathcal{A}_1 \to \mathcal{B}^+$, be a decomposition of~$\sigma$ as in Lemma~\ref{l:reducetotallength}.
If $|\mathcal{A}_1| < |\mathcal{A}|$, then we are done. 
Otherwise, we have $|\mathcal{A}_1| = |\mathcal{A}|$ and $\tau_1$ is injective on~$\mathcal{A}^\mathbb{N}$, thus $\sigma_1$ is not injective on~$\mathcal{A}_1^\mathbb{N}$.
Recursively, we write $\sigma_k = \sigma_{k+1} \tau_{k+1}$ with $\tau_{k+1}:\, \mathcal{A}_k \to \mathcal{A}_{k+1}^+$, $\sigma_{k+1}:\, \mathcal{A}_{k+1} \to \mathcal{B}^+$, as in Lemma~\ref{l:reducetotallength}, as long as $|\mathcal{A}_k| = |\mathcal{A}|$.
Since $\interleave\sigma\interleave > \interleave\sigma_1\interleave > \interleave\sigma_2\interleave > \cdots > \interleave\sigma_k\interleave$, we have $|\mathcal{A}_k| < |\mathcal{A}|$ for some $k \ge 1$. 
Let $\tilde{\sigma} = \sigma_k$ and $\tau =  \tau_k \tau_{k-1} \cdots \tau_1$.
Then we have $\sigma = \tilde{\sigma} \tau$, $|\tilde{\mathcal{A}}| < |\mathcal{A}|$ and $\interleave\tilde{\sigma}\interleave < \interleave\sigma\interleave$.
Since, for all $0 \le i < k$, each $\sigma_{i+1}(a)$, $a \in \mathcal{A}_{i+1}$, is a prefix of some $\sigma_i(b)$, $b \in \mathcal{A}$, with $\sigma_0 = \sigma$, we also have that $\tilde{\sigma}(a)$, $a \in \tilde{\mathcal{A}}$, is a prefix of some $\sigma(b)$, $b \in \mathcal{A}$. 
(If $\sigma$ is not injective on~$\mathcal{A}^{-\mathbb{N}}$, then each $\tilde{\sigma}(a)$, $a \in \tilde{\mathcal{A}}$, is a suffix of some $\sigma(b)$, $b \in \mathcal{A}$.)
\end{proof}

\begin{lemma} \label{l:injectiveinvertible}
Let $\sigma:\, \mathcal{A} \to \mathcal{B}^+$ be a morphism.
If $\mathrm{rk}(M_\sigma) = |\mathcal{A}|$, or if $|\mathcal{A}| = 2$ and $|\sigma(\mathcal{A}^\mathbb{Z})| \ge 2$, then $\sigma$ is injective on~$\mathcal{A}^\mathbb{Z}$. 
\end{lemma}

\begin{proof}
Suppose that $\sigma$ is not injective on $\mathcal{A}^\mathbb{Z}$, and write $\sigma = \tilde{\sigma} \tau$ as in Lemma~\ref{l:reducealphabet}.
Then we have $\mathrm{rk}(M_\sigma) \le \mathrm{rk}(M_{\tilde{\sigma}}) \le |\tilde{\mathcal{A}}| < |\mathcal{A}|$. 
If $|\mathcal{A}| = 2$, then we have $|\tilde{\mathcal{A}}| = 1$, thus $|\sigma(\mathcal{A}^\mathbb{Z})| = 1$.
\end{proof}

Note that if $\sigma$ is not fully recognizable for aperiodic points, then $\sigma(\mathcal{A}^\mathbb{Z})$ contains an aperiodic point and thus $|\sigma(\mathcal{A}^\mathbb{Z})| = \infty$.

\subsection{Proof of Theorem~\ref{t:recognizable}} \label{sec:proof-theor-reft:r}
If $\sigma$ is a left or right permutative morphism, then $\sigma$ is fully recognizable for aperiodic points by Lemma~\ref{l:permutative}. 
By Lemma~\ref{l:conjugaterecognizable}, this also holds when $\sigma$ is rotationally conjugate to such a morphism.

Suppose that there exists a morphism $\sigma:\, \mathcal{A} \to \mathcal{B}^+$ with $\mathrm{rk}(M_\sigma) = |\mathcal{A}|$ or $|\mathcal{A}| = 2$ that is not fully recognizable for aperiodic points.
From the set of morphisms with these properties, choose one with minimal total length.
Then, since $\sigma$ is injective on~$\mathcal{A}^\mathbb{Z}$ by Lemma~\ref{l:injectiveinvertible}, the lack of full recognizability implies that there exist $x, x' \in \mathcal{A}^\mathbb{N}$ and $a' \in \mathcal{A}$ such that $\sigma(x) = w\, \sigma(x')$ for some proper suffix~$w$ of $\sigma(a')$ with $0 < |w| < |\sigma(x_0)|$; let $\sigma(a') = vw$.
Then we have $\sigma = \sigma_1 \tau_1$ with $\tau_1:\, \mathcal{A} \to \mathcal{A}_1^+$, $\sigma_1:\, \mathcal{A}_1 \to \mathcal{B}^+$, $\mathcal{A}_1 = \mathcal{A} \cup \{a''\}$, where $a''$ is a letter that is not in~$\mathcal{A}$, $\tau_1(a') = a' a''$, $\tau_1(a) = a$ otherwise, $\sigma_1(a') = v$, $\sigma_1(a'') = w$, $\sigma_1(a) = \sigma(a)$ otherwise. 
Note that $\interleave\sigma_1\interleave = \interleave\sigma\interleave$, and each $\sigma_1(a)$, $a \in \mathcal{A}_1$, is a prefix of $\sigma(b)$ for some $b \in \mathcal{A}$. 
As $\sigma_1$ is not injective on $\mathcal{A}_1^\mathbb{N}$, we have a decomposition $\sigma_1 = \sigma_2 \tau_2$ by Lemma~\ref{l:reducealphabet}, with $\tau_2:\, \mathcal{A}_1 \to \mathcal{A}_2^+$, $\sigma_2:\, \mathcal{A}_2 \to \mathcal{B}^+$, $|\mathcal{A}_2| < |\mathcal{A}_1|$, $\interleave\sigma_2\interleave < \interleave\sigma_1\interleave$, and each $\sigma_2(a)$, $a \in \mathcal{A}_2$, is a prefix of some $\sigma_1(b)$, $b \in \mathcal{A}_1$. 
Therefore, we have 
\[
\sigma = \sigma_2 \tau_2 \tau_1, \quad \interleave\sigma_2\interleave < \interleave\sigma\interleave, \quad \mbox{and each $\sigma_2(a)$, $a \in \mathcal{A}_2$, is a prefix of some $\sigma(b)$, $b \in \mathcal{A}$}.
\]
Since $|\mathcal{A}_2| \le |\mathcal{A}|$, $\mathrm{rk}(M_\sigma) = |\mathcal{A}|$ implies that $\mathrm{rk}(M_{\tau_2\tau_1}) = \mathrm{rk}(M_{\sigma_2}) = |\mathcal{A}_2| = |\mathcal{A}|$, and $|\mathcal{A}| = 2$ implies that $|\mathcal{A}_2| = 2$. 
By the minimality of~$\interleave\sigma\interleave$, $\sigma_2$~is fully recognizable for aperiodic points.
If $|\sigma_2(a)| \ge 2$ for some $a \in \mathcal{A}_2$, then we also have $\interleave\tau_2\tau_1\interleave < \interleave\sigma\interleave$ and, hence, $\tau_2\tau_1$~is fully recognizable for aperiodic points.
By Lemma~\ref{lem:tele}, $\sigma$~is also fully recognizable for aperiodic points, contradicting our assumption. 
If $|\sigma_2(a)| = 1$ for all $a \in \mathcal{A}_2$, then $\mathrm{rk}(M_{\sigma_2}) = |\mathcal{A}_2| = |\mathcal{A}|$ and the fact that each $\sigma_2(a)$ is a prefix of some $\sigma(b)$, $b \in \mathcal{A}$, imply that $\sigma$ is left permutative, thus fully recognizable for aperiodic points, contradicting again our assumption. 
This concludes the proof of Theorem~\ref{t:recognizable}.

\section{$S$-adic shifts and recognizability} \label{main_S_adic}

Throughout this section, let $\boldsymbol{\sigma} = (\sigma_n)_{n\ge0}$ be a sequence of morphisms with $\sigma_n:\, \mathcal{A}_{n+1}\to \mathcal{A}_n^+$. Our main recognizability results for sequences of morphisms are Theorems~\ref{c:rec} and~\ref{t:evrec}.
We also prove an extended version of the stationary case, $\sigma_n=\sigma $ for each $n$,  in Theorem~\ref{t:substrec}.

We first recall basic definitions concerning $S$-adic shifts and define recognizability for sequences of morphisms.
$S$-adic shifts are obtained by replacing the iteration of a single substitution by the iteration of a sequence of morphisms.
There are two main ways to define a shift associated with such a sequence, as described in \cite{AubSab}. 
The first one, that we choose here, is by taking two-sided points whose subwords are all generated by iterating the morphisms.  Given a sequence of morphisms $\boldsymbol{\sigma} = (\sigma_n)_{n\ge0}$, for $0\leq n<N$, let 
\[
\sigma_{[n,N)} = \sigma_n \circ \sigma_{n+1} \circ \dots \circ \sigma_{N-1}.
\]
For $n\geq 0$, the \emph{languages $\mathcal{L}_{\boldsymbol{\sigma}}^{(n)}$ associated with~$\boldsymbol{\sigma}$} are defined by 
\[
\mathcal{L}_{\boldsymbol{\sigma}}^{(n)} = \big\{w \in \mathcal{A}_n^*:\, \mbox{$w$ is a subword of $\sigma_{[n,N)}(a)$ for some $a \in\mathcal{A}_N$, for some $N>n$}\big\}.
\]
For each $n\geq 0$ let $X_{\boldsymbol{\sigma}}^{(n)}$ be the set of points $x \in \mathcal{A}_n^\mathbb{Z}$ all of whose subwords belong to~$  \mathcal{L}_{\boldsymbol{\sigma}}^{(n)}$. We consider
$(X_{\boldsymbol{\sigma}}^{(n)},T)$, the {\em shift generated by~$\mathcal{L}_{\boldsymbol{\sigma}}^{(n)}$}. 
For ease of notation, we set $X_{\boldsymbol{\sigma}} = X_{\boldsymbol{\sigma}}^{(0)}$ and call $(X_{\boldsymbol{\sigma}},T)$ the \emph{$S$-adic shift} generated by the \emph{directive sequence}~$\boldsymbol{\sigma}$.

Note that the language of~$X_{\boldsymbol{\sigma}}^{(n)}$ might be strictly included in~$\mathcal{L}_{\boldsymbol{\sigma}}^{(n)}$: There might be words in the language~$\mathcal{L}_{\boldsymbol{\sigma}}^{(n)}$ that occur in no two-sided point in~$X_{\boldsymbol{\sigma}}^{(n)}$. 
This issue already arises in the non-primitive substitutive case. For example, take the case where $a$ only occurs as a prefix of  $\sigma(a)$; then there is no left-extension of~$a$. 

A~second definition of an $S$-adic shift is that where one considers the shift generated  by the set of {\em limit words} $\bigcap_{n\in\mathbb{N}} \sigma_{[0,n)}(\mathcal{A}_n ^\mathbb{Z})$. Note that these two definitions generally yield different shifts, already in  the non-minimal substitutive case. Consider for example the  substitution $\sigma:\, 0 \mapsto 00,\, 1 \mapsto 11$, on the alphabet $\{0,1\}$ as an illustration (see also \cite{AubSab}), and the constant sequence $\boldsymbol{\sigma}$ taking the constant value~$\sigma$.
The point $\cdots0011\cdots$ does not belong to~$X_{\boldsymbol{\sigma}}$, but does belong to  $\bigcap_{n\in\mathbb{N}} \sigma^n (\{0,1\}^\mathbb{Z})$. 
Lemma \ref{l:Xlimitword} below illustrates the fact that both shifts might differ.

We choose to adopt the first  definition, i.e., that of the shift generated by the languages~$\mathcal{L}_{\boldsymbol{\sigma}}^{(n)}$, since we consider this definition as being  closer to the usual definition of a substitutive shift, and also, we avoid situations where letters are artificially glued. 
However, we will  make use of limit words in Section~\ref{sec:minimal}.

\subsection{Recognizability for $S$-adic shifts}\label{definition-of-S-adic-shift}

\begin{definition}[Recognizable sequence of morphisms] \label{recognizable-morphism-sequences}
A~directive sequence $\boldsymbol{\sigma}$ is \emph{recognizable at level~$n$} if $\sigma_n$ is recognizable in~$X_{\boldsymbol{\sigma}}^{(n+1)}$. 
The sequence~$\boldsymbol{\sigma}$ is \emph{recognizable} if it is recognizable at level~$n$ for each $n \ge 0$; if there is an $n_0 \in \mathbb{N}$ such that $\boldsymbol{\sigma}$ is recognizable at level~$n$ for each $n \ge n_0$, then we say that  $\boldsymbol{\sigma}$ is \emph{eventually recognizable}. 
We will use all these notions \emph{for aperiodic points} as well. 
\end{definition}

The following lemma tells us that for each~$n$, every point~$x$ in~$X_{\boldsymbol{\sigma}}^{(n)}$ admits at least one desubstitution using~$\sigma_n$ and a point in~$X_{\boldsymbol{\sigma}}^{(n+1)}$. 
Therefore, recognizability at level~$n$ says that each element of~$X_{\boldsymbol{\sigma}}^{(n)}$ has \emph{exactly one} $\sigma_n$-representation in~$X_{\boldsymbol{\sigma}}^{(n+1)}$.

\begin{lemma} \label{at_least_one}
Each element of~$X_{\boldsymbol{\sigma}}^{(n)}$ has a (centered) $\sigma_{[n,N)}$-representation in~$X_{\boldsymbol{\sigma}}^{(N)}$ for all $N > n$. 
In particular, we have $X_{\boldsymbol{\sigma}}^{(n)} = \bigcup_{k\in\mathbb{Z}} T^k \sigma_n(X_{\boldsymbol{\sigma}}^{(n+1)})$ for all $n \ge 0$.  
\end{lemma}

\begin{proof}
Let $y \in X_{\boldsymbol{\sigma}}^{(n)}$.
By the definition of~$X_{\boldsymbol{\sigma}}^{(n)}$, each word $y_{[-\ell,\ell)}$ is a subword of~$\sigma_{[n,N')}(a)$ for some $a \in \mathcal{A}_{N'}$, $N'> n$, and we have $N' > N$ if $\ell$ is large.
Then $y_{[-\ell+k,\ell-k')} = \sigma_{[n,N)}(w)$ for some $w \in \mathcal{L}_{\boldsymbol{\sigma}}^{(N)}$, $0 \le k,k' < \max_{a\in\mathcal{A}_N} |\sigma_{[n,N)}(a)|$.
Since $|w| \to \infty$ as $\ell \to \infty$, a Cantor diagonal argument gives a word $x \in X_{\boldsymbol{\sigma}}^{(N)}$ and $0 \le k < |\sigma_{[n,N)}(x_0)|$ such that $(k,x)$ is a $\sigma_{[n,N)}$-representation of~$y$.
\end{proof}

\subsection{Examples of only eventual recognizability, and non-eventual recognizability}\label{examples}

 We say the directive sequence~$\boldsymbol{\sigma}$ is \emph{primitive} if for each $n\ge0$, the incidence matrix of $\sigma_{[n,N)}$ is a positive matrix for \emph{some} $N> n$.
If $\boldsymbol{\sigma}$ is primitive, then $(X_{\boldsymbol{\sigma}}^{(n)},T)$ is minimal for all~$n$, by \cite[Lemma~5.2]{Berthe-Delecroix}. In fact the following examples are all primitive, so lie squarely within the mainstream symbolic dynamics framework, where minimality is a common assumption.

First we give an example of a directive sequence that is eventually recognizable but not  recognizable. 

\begin{example} \label{ex:ws}
Let $\theta:\, \mathcal{A}_0 \to \mathcal{A}_0^+$ be an aperiodic primitive substitution of odd constant length such that, for some $a \in \mathcal{A}_0$, $\theta(a)$ starts with $a$ and $\theta^n(a)$ starts with $wa$ for some $n \ge 1$, $w \in \mathcal{A}_0^+$ of odd length; we illustrate this example by taking $\mathcal{A}_0=\{0,1\}$, $\theta(0) = 00100$, $\theta(1)= 00000$.  Note that 
 $ \mathcal{L}_\theta \cap \mathcal{A}_0^2 =\{ 00, 01, 10 \}$. Let  $\mathcal{A}_1 = \{A,B,C\}$ and let  
 $\sigma_0:\, \mathcal{A}_1 \to \mathcal{A}_0^+$ be the  morphism  $\sigma_0(A) = 00$, $\sigma_0(B) = 01$, $\sigma_0(C)= 10$. 
Then the substitution $\theta_1:\, \mathcal{A}_1 \to \mathcal{A}_1^+$ given by $\theta_1 = \sigma_0^{-1} \circ \theta \circ \sigma_0$ is well-defined and of the same constant length as~$\theta$; for our choice of~$\theta$, we have $\theta_1(A) = ACABA$, $\theta_1(B) = ACAAA$, $\theta_1(C) = AAABA$, thus our $\theta_1$ is {\em proper}:
Recall that a~morphism $\sigma:\, \mathcal{A} \to \mathcal{B}^+$ is proper if there exist two letters $b,e \in \mathcal{A}$ such that, for all $a \in \mathcal{A}$, $\sigma(a)$ begins with~$b$ and ends with~$e$. 
To see that $\theta_1$ is primitive in general, note that, by primitivity of~$\theta$, there exists $m \ge 1$ such that $\theta^m(b)$ contains $a$ for each $b \in \mathcal{A}_0$ and $\theta^m(a)$ contains all elements of $\mathcal{L}_\theta \cap \mathcal{A}_0^2$ as subword.
Moreover, $\theta^{m+n}(a)$ starts with $\theta^m(a)$ as well as with $\theta^m(w) \theta^m(a)$, where $\theta^m(w)$ is of odd length, thus each element of $\mathcal{L}_\theta \cap \mathcal{A}_0^2$ occurs at even and odd positions in~$\theta^{m+n}(a)$.
Hence $\theta_1^{2m+n}(b)$ contains all elements of $\mathcal{A}_1$ for all $b \in \mathcal{A}_1$. 

Let $\boldsymbol{\sigma} = (\sigma_n)_{n\ge0}$ with $\sigma_n = \theta_1$ for all $n\geq 1$. 
Then we have $X_{\boldsymbol{\sigma}} = X_\theta$ because $\sigma_0 \circ \theta_1^n = \theta^n \circ \sigma_0$ for all $n \ge 0$, and $X_{\boldsymbol{\sigma}}^{(n)} = X_{\theta_1}$ for all $n \ge 1$. 
In particular, the aperiodicity of $(X_\theta,T)$ implies that $(X_{\theta_1},T)$ is aperiodic.
As for $\mathcal{L}_\theta \cap \mathcal{A}_0^2$, we obtain that each word in~$\mathcal{L}_\theta$ occurs at even and odd positions in $\theta^m(a)$ for sufficiently large~$m$.
Therefore, each word in~$\mathcal{L}_\theta$ of even length is in $\sigma_0(\mathcal{L}_{\theta_1})$ and, hence, $X_\theta \subseteq \sigma_0(X_{\theta_1})$.
Since $\sigma_0(X_{\boldsymbol{\sigma}}^{(1)}) \subseteq X_{\boldsymbol{\sigma}}$, we get that $X_{\boldsymbol{\sigma}} = \sigma_0(X_{\boldsymbol{\sigma}}^{(1)})$. 
In particular, we have $T \sigma_0(X_{\boldsymbol{\sigma}}^{(1)}) = \sigma_0(X_{\boldsymbol{\sigma}}^{(1)})$, thus $\sigma_0$ is not recognizable in~$X_{\boldsymbol{\sigma}}^{(1)}$.
By Moss\'{e}'s theorem \cite{Mosse:96}, $\theta_1$ is recognizable in~$X_{\theta_1}$, thus $\sigma_n$ is recognizable in~$X_{\boldsymbol{\sigma}}^{(n+1)}$ for all $n \ge 1$. 
Hence, we have an eventually recognizable but not recognizable sequence~$\boldsymbol{\sigma}$, with $\boldsymbol{\sigma}$ being moreover primitive and such that all~$\sigma_n$, $n\ge 1$, are proper. The relevance of this is that the natural Bratteli-Vershik system that $\boldsymbol{\sigma}$ defines is a topological system which is not conjugate to $(X_{\boldsymbol{\sigma}}, \sigma)$; see Section~\ref{sec:bratt-versh-repr} for details.

\end{example} 

One can extend this example further to get (primitive) directive sequences $\boldsymbol{\sigma}$ that are not eventually recognizable. 
  
\begin{example} \label{ex:ws2}
We repeat the above procedure inductively. 
Let $\theta = \theta_0:\, \mathcal{A}_0 \to \mathcal{A}_0^+$ be an aperiodic primitive substitution of odd constant length such that, for some $a \in \mathcal{A}_0$, $\theta_0(a)$ starts with $a$ and $\theta_0^m(a)$ starts with $wa$ for some $m \ge 1$, $w \in \mathcal{A}_0^*$ of odd length.
Let $\sigma_0$, $\theta_1$ and~$\mathcal{A}_1$ be as in Example~\ref{ex:ws}. 
Then $\theta_1$ is an aperiodic primitive substitution of odd constant length.
For $b \in \mathcal{A}_1$ such that $\sigma_0(b)$ is a prefix of $\theta_0(a)$, $b$ is a prefix of $\theta_1(b)$ and $\sigma_0 \theta_1^{2m+2}(b)$ starts with $\theta_0^{2m+2}(a)$, thus with $\theta_0^{m+2}(w) \theta_0(wa)$ as well as $\theta_0^{m+1}(w) \theta_0(wa)$. 
For $|\theta_0(a)| = k$, we have $|\theta_0^{m+2}(w) \theta_0(w)| = (k^{m+2}+k) |w|$ and $|\theta_0^{m+2}(w) \theta_0(w)| = (k^{m+1}+k) |w|$.
As $k$ and $|w|$ are odd, at least one of these numbers is an odd multiple of~$2$. 
Therefore, $\theta_1^{2m+2}(b)$ starts with $vb$ for some $v \in \mathcal{A}_1^*$ of odd length. 
Therefore, we define recursively for $n \ge 0$, an alphabet~$\mathcal{A}_{n+1}$ that is in bijective correspondence, via $\sigma_n$, with the set of words of length two in~$\mathcal{L}_{\theta_n}$, and $\theta_{n+1} = \sigma_n^{-1} \circ \theta_n \circ \sigma_n$.
Let $\boldsymbol{\sigma} = (\sigma_n)_{n\geq 0}$.
Then $X_{\theta_n} = X_{\boldsymbol{\sigma}}^{(n)} = \sigma_n(X_{\boldsymbol{\sigma}}^{(n+1)})$ for all $n \ge 0$, thus $\boldsymbol{\sigma}$ is not recognizable at any level~$n$. 
Here, $\boldsymbol{\sigma}$~is primitive.
In this example, the size of the alphabets $\mathcal{A}_n$ grows (exponentially).
\end{example}

Summing up, we have proved  that Moss\'e's theorem cannot be  extended in a straightforward manner to the $S$-adic framework.  We stress the fact that the morphisms involved in the sequences $\boldsymbol{\sigma}$ above are defined on alphabets of various sizes. 
In Example~\ref{ex:ws}, we could replace $\sigma_0$ by a substitution on~$\mathcal{A}_1$ and have thus constant alphabets, e.g.\ take $\sigma_0(A) = AA$, $\sigma_0(B) = ABC$, $\sigma_0(C) = BCA$.
In Example~\ref{ex:ws2}, this is not possible. 
We state this in the following proposition.

\begin{proposition} \label{prop:prim}
There exist primitive sequences of substitutions that are eventually recognizable but not recognizable.
There exist primitive sequences of morphisms that are not eventually recognizable.
\end{proposition}

\subsection{First recognizability results}
As an immediate consequence of Theorem \ref{t:recognizable}, we have the following.

\begin{theorem}\label{c:rec}
Let $\boldsymbol{\sigma} = (\sigma_n)_{n\ge0}$ be a sequence of morphisms with $\sigma_n:\, \mathcal{A}_{n+1}\to \mathcal{A}_n^+$.
If each morphism $\sigma_n$ satisfies one of 
\begin{itemize}
\item  $\mathrm{rk}(M_{\sigma_{n}}) = |\mathcal{A}_{n+1}|$, or
\item $|\mathcal{A}_{n+1}| = 2$, or
\item  $\sigma_n$ is (rotationally conjugate to) a left or right permutative morphism,
\end{itemize}
then $\boldsymbol{\sigma}$ is  recognizable for aperiodic points.
\end{theorem}

Note that if $\sigma:\mathcal{A} \rightarrow \mathcal{B}$ is permutative or if $\mathrm{rk}(M_{\sigma}) = |\mathcal{A}|$, then  $|\mathcal{B}|\geq |\mathcal{A}|$. 
This implies that all considered directive sequences in Theorem~\ref{c:rec} are defined on bounded alphabets.

Next we give examples of classical $S$-adic shifts that are recognizable as direct consequence of Theorem~\ref{c:rec}. Indeed  these examples are associated with continued fractions and involve substitutions with unimodular incidence matrices (they are square matrices and have determinant~$\pm1$, thus they have full rank).

\begin{example}[Arnoux-Rauzy shifts]\label{exa:AR}
Let $\mathcal{A}=\{1,2,\ldots,d\}$. 
Arnoux-Rauzy substitutions are defined as substitutions of the form $\mu_i:\, i \mapsto i,\, j \mapsto ji\ \mbox{for}\ j \in \mathcal{A} \setminus \{i\}$.
An Arnoux-Rauzy word~\cite{Arnoux-Rauzy:91} is a two-sided sequence in the shift $(X_{\boldsymbol{\mu}},T)$ generated by any directive sequence of Arnoux-Rauzy  substitutions $\boldsymbol{\mu}=(\mu_{i_n})$, where the sequence $(i_n)_{n\geq 0} \in \mathcal{A}^\mathbb{N}$ is such that every letter in~$\mathcal{A}$ occurs infinitely often in~$(i_n)_{n\ge0}$. 
For more on Arnoux-Rauzy sequences see e.g.\ \cite{Cassaigne-Ferenczi-Zamboni:00,Cassaigne-Ferenczi-Messaoudi:08}. 
In the case $d=2$, one recovers Sturmian sequences. 
Recognizability of $\boldsymbol{\mu}$ can be proved in various ways: 
It is easy to see that each $\mu_i$ is fully recognizable. 
It is also a consequence of Theorem~\ref{c:rec} since each $\mu_i$ is left permutative (and also rotationally conjugate to a right permutative substitution, and also its incidence matrix has full rank), as $(X_{\boldsymbol{\mu}},T)$ is aperiodic by the assumptions on~$(i_n)$.
\end{example}

\begin{example}[Unimodular continued fractions]
As explained in \cite{Berthe:11,Berthe-Delecroix,Berthe-Ferenczi-Zamboni:05},   $S$-adic expansions  are closely  related to continued fraction expansions.  
Indeed,  usual 
  multidimensional continued  fraction algorithms  produce  matrices with nonnegative entries;  we  then consider these matrices  as
incidence matrices of accordingly defined substitutions. Note that the choice of substitution in not canonical.
In other words,  a continued fraction algorithm  produces  directive sequences, and thus  $S$-adic  shifts.
The connection between a continued fraction algorithm and the associated $S$-adic shift then runs via frequencies: an expansion of the continued fraction algorithm produces a sequence of matrices and, hence, a sequence of the associated substitutions. These define an $S$-adic shift which has the property that its letter frequency vector (under suitable assumptions that provide its existence) admits this particular continued fraction expansion. The continued fraction algorithm here acts  as a renormalization process.  A fundamental example of this relation is between Sturmian sequences and regular continued fractions; see for instance~\cite{Arnoux-Fisher:01}. An illustration of this relation in terms Brun substitutions is considered in \cite{Berthe-Steiner-Thuswaldner}. Recognizability of all these $S$-adic shifts follows from Theorem~\ref{c:rec} by  the unimodularity of the matrices.

Moreover, minimal  shifts  having a sublinear number of subwords of  a given length
 are known to be $S$-adic \cite{Ferenczi:96},  and the $S$-adic expansion can be seen as  renormalization process.
 Furthermore, several induction and renormalization  procedures also yield suitable  $S$-adic expansions involving  unimodular matrices. The emblematic case is provided by aperiodic shifts generated by natural codings of an interval exchange whose orbits of discontinuities are infinite and disjoint, that are known 
to be $S$-adically generated by  directive  sequences of substitutions obtained   by applying Rauzy induction  (or any of its accelerations such as Rauzy-Veech-Zorich induction \cite{Zorich}, or  Marmi-Moussa-Yoccoz induction \cite{MarMouYoc}). Here again recognizability applies by unimodularity. We summarize this application of Theorem \ref{c:rec}  as follows.
\end{example}

\begin{proposition}
Let $\boldsymbol{\sigma}$ be a directive sequence obtained from a unimodular continued fraction expansion algorithm. 
Then $\boldsymbol{\sigma}$ is recognizable.
\end{proposition}

\section{Eventual recognizability of $S$-adic shifts on bounded alphabets} \label{main_S_adic-2}

\subsection{Statement of the result} 
In this section, we prove eventual recognizability under mild conditions. 

\begin{theorem} \label{t:evrec}
Let $\boldsymbol{\sigma} = (\sigma_n)_{n\ge0}$ be a sequence of morphisms with $\sigma_n:\, \mathcal{A}_{n+1}\to \mathcal{A}_n^+$ such that $\liminf_{n\to\infty} |\mathcal{A}_n| < \infty$ and $\sup_{n\ge0} |\{\mathcal{L}_x:\, x \in X_{\boldsymbol{\sigma}}^{(n)}\}| < \infty$. 
Then $\boldsymbol{\sigma}$ is eventually recognizable for aperiodic points. 
\end{theorem}

For a minimal shift $(X,T)$, each $x \in X$ has the language $\mathcal{L}_x = \mathcal{L}_X$, i.e., $|\{\mathcal{L}_x:\, x \in X\}| = 1$. 
By Proposition~\ref{p:minimalcomponents} below, the number of different languages in~$X_{\boldsymbol{\sigma}}^{(n)}$ is also bounded for \emph{everywhere growing} directive sequences~$\boldsymbol{\sigma}$ with $\liminf_{n\to\infty} |\mathcal{A}_n| < \infty$, where everywhere growing means that $\lim_{n\to \infty}\min_{a\in \mathcal{A}_n} |\sigma_{[0,n)}(a)| = \infty$.
Note that  $(X_{\boldsymbol{\sigma}},T)$ has zero entropy in this case by \cite[Theorem~4.3]{Berthe-Delecroix}. We  formulate this special instance of Theorem \ref{t:evrec}.

\begin{theorem} \label{t:evrec-1}
Let $\boldsymbol{\sigma} = (\sigma_n)_{n\ge0}$ be an everywhere growing  sequence of morphisms $\sigma_n:\, \mathcal{A}_{n+1}\to \mathcal{A}_n^+$ such that $\liminf_{n\to\infty} |\mathcal{A}_n| < \infty$.
Then $\boldsymbol{\sigma}$ is eventually recognizable for aperiodic points. 
\end{theorem}

Moreover, we give bounds for the number of levels where $\boldsymbol{\sigma}$ is not recognizable, depending only on the size of the alphabets and the number of different languages in~$X_{\boldsymbol{\sigma}}^{(n)}$.

For a substitution~$\sigma$, we show in Proposition~\ref{p:sigmacomp} below that $\{\mathcal{L}_x:\, x \in X_\sigma\}$ is always a finite set.
As eventual recognizability of a stationary sequence~$\boldsymbol{\sigma}$ is equivalent to recognizability, we obtain the following theorem, which improves a result of Bezuglyi, Kwiatkowski and Medynets \cite[Theorem~5.17]{Bezugly:2009} stating that each aperiodic substitution~$\sigma$ is recognizable in~$X_\sigma$.

\begin{theorem} \label{t:substrec}
Let $\sigma$ be a substitution. 
Then $\sigma$ is recognizable in~$X_\sigma$ for aperiodic points. 
\end{theorem}

Our proof is based on the work of Downarowicz and Maass in \cite{Down:2008}. 
This work has already been modified in \cite{Bezugly:2009}; our work can thus be seen as an extension of these last two articles. 
We first reduce Theorem~\ref{t:evrec} to the slightly more special Proposition~\ref{p:evrec2}, which will be proved in Sections~\ref{sec:nocu} and~\ref{sec:proof2}, first in the special case where each shift $(X_{\boldsymbol{\sigma}}^{(n)},T)$ is minimal and each $\sigma_n$ is injective on~$X_{\boldsymbol{\sigma}}^{(n+1)},$ before considering the general case. 

We point out some (necessary) differences in our proof strategy. 
In \cite{Down:2008} the authors work with one topological dynamical system $(X,F)$: a topological Bratteli-Vershik system, taking a sequence of shift factors $(X_i,T)$ of the given system which (collectively) separate points. 
Bezuglyi, Kwiatkowski and Medynets work with a stationary aperiodic $S$-adic shift, i.e., one where $\sigma_n= \sigma_0$ for each $n\geq 0$. 
This entails that their spaces are stationary, i.e., $X_{\boldsymbol{\sigma}}^{(n)} = X_{\boldsymbol{\sigma}}$ for each $n\geq 0$. 
In \cite{Down:2008}, the strategy was to obtain a sufficiently large collection of points in~$X$ that project to the same point in a fixed factor $(X_i,T)$, and whose projections do not synchronize, i.e., have no common cuts. 
This lack of synchronicity implies that $X_i$ is finite, using the Infection Lemma (Lemma~\ref{l:infection}).
In \cite{Bezugly:2009}, a similar strategy is followed to create points with arbitrarily many $\sigma$-representations.
To obtain this collection of points, both sets of authors use limiting arguments in their (fixed) space.
Here our spaces change, and so we need another mechanism to ensure that we have sufficiently many points in some~$X_{\boldsymbol{\sigma}}^{(n)}$ in order to invoke the Infection Lemma. 
We overcome this issue by showing that if one point~$x$ has at least two (special kinds of) centered representations, then so does every point whose language is contained in that of~$x$ (Lemma~\ref{l:oneforall}). 
Using this fact, we are able to construct points in~$X_{\boldsymbol{\sigma}}$ which have as many  (non-synchronized) $\sigma_{[0,n)}$-representations as we need, provided that $n$ is large. 

\subsection{Telescoping an $S$-adic sequence}\label{sec:telesc}
We can consider a \emph{telescoping} of a directive sequence of morphisms. 
Namely, given a directive sequence $ {\boldsymbol{\sigma}}= (\sigma_n)_{n\geq 0}$ and an increasing sequence of integers $(n_k)_{k \geq 0}$ with $n_0=0$, we define the telescoping of $\boldsymbol{\sigma}$ along  $(n_k)_{k \geq 0}$ to be the directive sequence  $\boldsymbol{\sigma}'= (\sigma'_k)_{k\geq 0}$ where $\sigma'_k = \sigma_{[n_k,n_{k+1})}$. 
Note that $X_{\boldsymbol{\sigma}'} = X_{\boldsymbol{\sigma}}$ and $X_{\boldsymbol{\sigma}'}^{(k)} = X_{\boldsymbol{\sigma}}^{(n_k)}$ for each~$k$.  

By Lemma~\ref{lem:tele}, $\boldsymbol{\sigma}$ is (eventually) recognizable if and only if $\boldsymbol{\sigma}'$ is (eventually) recognizable.  The situation is more complicated when we consider recognizability for aperiodic points. For, let $\sigma:X\rightarrow Y$ and $\tau:Y\rightarrow Z$. If $Z$ contains periodic points which have aperiodic $\tau$-representations in~$Y$, then it is possible that $\sigma$ is not recognizable in $X$ for aperiodic points, even if $\tau \sigma$ is recognziable in~$X$ for aperiodic points. However,
when $|\{\mathcal{L}_x:\, x \in X_{\boldsymbol{\sigma}}^{(n)}\}|$ is bounded, there are only finitely many $n$  where $\sigma_n(x)$ is periodic for some aperiodic $x \in X_{\boldsymbol{\sigma}}^{(n+1)}$, and 
we have that $\boldsymbol{\sigma}$ is eventually recognizable for aperiodic points if and only if $\boldsymbol{\sigma}'$ is eventually recognizable for aperiodic points.
Therefore, telescoping enables us to replace the assumption $\liminf_{n\to\infty} |\mathcal{A}_n| < \infty$ in Theorem~\ref{t:evrec} by the assumption that $|\mathcal{A}_n|$ is bounded. 
Thus Theorem~\ref{t:evrec} is proved if we establish the following proposition.

\begin{proposition} \label{p:evrec2}
Let $\boldsymbol{\sigma} = (\sigma_n)_{n\ge0}$ be a sequence of morphisms with $\sigma_n:\, \mathcal{A}_{n+1}\to \mathcal{A}_n^+$ such that $|\mathcal{A}_n|$ and $|\{\mathcal{L}_x:\, x \in X_{\boldsymbol{\sigma}}^{(n)}\}|$ are bounded. 
Then $\boldsymbol{\sigma}$ is eventually recognizable for aperiodic points. 
\end{proposition}

We will even show that there are less than $(K-1) ((K-1+\lfloor \log_2(K-1) \rfloor)  L+1)$ levels where $\boldsymbol{\sigma}$ is not recognizable when $|\mathcal A_n|\leq K$ and $   |\{\mathcal{L}_x:\, x \in X_{\boldsymbol{\sigma}}^{(n)}\}\leq L$. 
Moreover, by Proposition~\ref{p:minimalcomponents} below, we have $L \le (K^2-3K+5)K/3$ when $\boldsymbol{\sigma}$ is everywhere growing.
If $\sigma_n$ is injective on $\mathcal{A}_{n+1}^\mathbb{Z}$ for all $n\ge0$, then we can omit the factor $K-1$, i.e., there are at most $(K-1+\lfloor \log_2(K-1) \rfloor)  L$ levels where $\boldsymbol{\sigma}$ is not recognizable.

\subsection{On $\sigma$-representations with no common $\sigma$-cut}\label{sec:nocu}
An important ingredient for the proof of Proposition~\ref{p:evrec2} is the following lemma, which is due to \cite{Down:2008}, except that Downarowicz and Maass only prove that the shift generated by~$y$ has periodic points and assume that $y$ has $|\mathcal{A}|^{|\mathcal{A}|+1}+1$ different $\sigma$-representations with pairwise no common $\sigma$-cut.
We give here a proof which is close to that of H{\o}ynes \cite{Hoynes:2017}.
Note that, contrary to a conjecture of H{\o}ynes, the bound $2^{|\mathcal{A}|-1} (|\mathcal{A}|-1) + 2$ in Lemma~\ref{l:infection} is not optimal, since, for $|\mathcal{A}| = 2$, each aperiodic $y \in \mathcal{B}^\mathbb{Z}$ has a unique centered $\sigma$-representation by Theorem~\ref{t:recognizable}.

\begin{lemma}[Infection Lemma] \label{l:infection}
Let $\sigma:\, \mathcal{A} \to \mathcal{B}^+$ be a morphism. 
If $y \in \mathcal{B}^\mathbb{Z}$ has $2^{|\mathcal{A}|-1} (|\mathcal{A}|-1) + 2$ different $\sigma$-representations with pairwise no common $\sigma$-cut, then $y$ is periodic. 
\end{lemma}

\begin{proof}
Let $K = |\mathcal{A}|$, $L = 2^{K-1} (K-1) + 2$, and let $(k^{(i)}, x^{(i)})$, $1 \le i \le L$, be $\sigma$-representations of~$y$ such that $C_\sigma(k^{(i)}, x^{(i)}) \cap C_\sigma(k^{(i')}, x^{(i')}) = \emptyset$ for all $i \ne i'$. 
For $n \in \mathbb{Z}$, $1 \le i \le L$, let $h_n^{(i)},\, m_n^{(i)} \in \mathbb{Z}$ be such that $h_n^{(i)}$ is the $m_n^{(i)}$-th $\sigma$-cutting point of $(k^{(i)}, x^{(i)})$ and $h_n^{(i)} \le n < h_n^{(i)} + |\sigma(x^{(i)}_{m_n^{(i)}})|$.

For $a \in \mathcal{A}$, let $\ell(a)$ be the least period length of $\sigma(a)$, i.e., if $\sigma(a) = z_1 z_2 \cdots z_{|\sigma(a)|}$, then $\ell(a) \ge 1$ is minimal such that $z_{j+\ell(a)} = z_j$ for all $1 \le j \le |\sigma(a)|-\ell(a)$; we have $\ell(a) \le |\sigma(a)|$.
Assume w.l.o.g.\ that $\mathcal{A} = \{1,2,\ldots,K\}$ and $\ell(1) \le \ell(2) \le \cdots \le \ell(K)$.
For $1 \le j \le K$, set
\[
I_{n,j} = \{1 \le i \le L:\, x^{(i)}_{m_n^{(i)}} = j\}.
\]

If $|I_{n,j}| \ge 2$, then we can choose $i, i' \in I_{n,j}$ with $h_n^{(i)} < h_n^{(i')}$, and we have $y_{\tilde{n}} = y_{\tilde{n}+\ell(j)}$ for all $h_n^{(i)} \le \tilde{n} < h_n^{(i')} + |\sigma(j)| - \ell(j)$.
Since $h_n^{(i)} \le h_n^{(i')} - \ell(j) \le n-\ell(j)$ and $h_n^{(i')} + |\sigma(j)| - \ell(j) \ge h_n^{(i)} + |\sigma(j)| > n$, we obtain that $y_{n-\ell(j)} = y_n = y_{n+\ell(j)}$; see also Figure~\ref{f:infection}.
Therefore, if for some $j \in \mathcal{A}$ we have $|I_{n,j}| \ge 2$ for all $n \in \mathbb{Z}$, then $T^{\ell(j)}(y) = y$ and thus $y$ is periodic. 

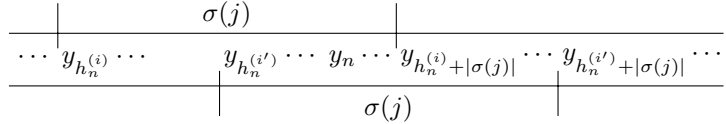
\begin{figure}[ht]
\begin{tikzpicture}
\draw(-4,0)--(5.5,0) (-4,.7)--(5.5,.7) (-3.35,.5)--(-3.35,1.1) (1.15,.5)--(1.15,1.1) (-1.2,.2)--(-1.2,-.4) (3.3,.2)--(3.3,-.4);
\node[above] at (-1.1,.65){$\sigma(j)$};
\node[above] at (-3,0){$\cdots\ y_{h_n^{(i)}} \cdots$};
\node[above] at (0,0){$y_{h_n^{(i')}} \cdots\ y_n\cdots$};
\node[above] at (2.25,0){$y_{h_n^{(i)}+|\sigma(j)|} \cdots$};
\node[above] at (4.45,0){$y_{h_n^{(i')}+|\sigma(j)|}\cdots$};
\node[below] at (1.05,0){$\sigma(j)$};
\end{tikzpicture}
\caption{Illustration of the case $i, i' \in I_{n,j}$ in the proof of Lemma~\ref{l:infection}.} \label{f:infection}
\end{figure}

Assume in the following that $|I_{n,K}| \le 1$ for some $n \in \mathbb{Z}$. 
Then we have 
\begin{equation} \label{e:naj}
|I_{n,j}| > 2^{j-1} (K-1) + 1
\end{equation}
for some $1 \le j < K$, as $\sum_{j=1}^K |I_{n,j}| \le \sum_{j=1}^{K-1} (2^{j-1} (K-1) + 1) + 1 = 2^{K-1} (K-1) + 1 < L = \sum_{j=1}^K |I_{n,j}|$ otherwise. 
Let $j$ be minimal such that \eqref{e:naj} holds for some $n \in \mathbb{Z}$.

We claim that $T^{\ell(a_j)}(y) = y$, contradicting that $y$ is aperiodic.  
To prove the claim, choose $n$ such that \eqref{e:naj} holds.
We show first that $y_{N} = y_{N+\ell(j)}$ for all $N \ge n$, by recursion on~$N$.
Assume that $y_{n'} = y_{n'+\ell(j)}$ for all $n \le n' < N$.
By the minimality of~$j$, we have
\[
\sum_{J=1}^{j-1} |I_{N,J}| \le \sum_{J=1}^{j-1} (2^{J-1} (K-1) + 1) = 2^j (K-1) +j-K,
\]
thus $|I_{n,j} \cap \bigcup_{J=j}^K I_{N,J}| > K-j+1$ and, hence, $|I_{n,j} \cap I_{N,J}| \ge 2$ for some $J \ge j$. 
For this~$J$, we have seen above that $y_{N-\ell(J)} = y_N$ and $y_{N+\ell(j)-\ell(J)} = y_{N+\ell(j)}$, as $0 \le \ell(j) \le \ell(J)$. 
In particular, we have $y_N = y_{N+\ell(j)}$ if $J = j$. 
If $J \ne j$, then $i \in I_{n,j} \cap I_{N,J}$ and $N \ge n$ imply that $h_N^{(i)} > n$, thus $n < N-\ell(J) < N$ and, hence, $y_{N-\ell(J)} = y_{N+\ell(j)-\ell(J)}$ by our assumption.
Therefore, we have $y_N = y_{N+\ell(j)}$ for all $N \ge n$.
By symmetry, this also holds for all $N < n$, i.e., $T^{\ell(j)}(y) = y$. 
\end{proof}

The proof of Proposition~\ref{p:evrec2} will be done by contradiction. 
Assuming non-recognizability for aperiodic points, we will construct many $\sigma$-representations with pairwise no common $\sigma$-cut, which contradicts Lemma~\ref{l:infection}. 
In this section, we will therefore establish a series of lemmas on $\sigma$-representations with pairwise no common $\sigma$-cut.

First we show that non-recognizable morphisms $\sigma$ that are injective on bi-infinite words give $\sigma$-representations with no common $\sigma$-cut.
To this end we use the following observation. 

\begin{lemma} \label{l:cuty}
Let $\sigma:\, \mathcal{A} \to \mathcal{B}^+$ be a morphism,  and let $(k,x), (k',x')$ be centered $\sigma$-representations of some $y \in \mathcal{B}^\mathbb{Z}$.
If $T^\ell(x) = T^{\ell'}(x')$ and the $\ell$-th $\sigma$-cutting point of $(k,x)$ coincides with the $\ell'$-th $\sigma$-cutting point of $(k',x')$, then $(k,x) = (k',x')$. 
\end{lemma}

\begin{proof}
Assume that $\ell \ge 0$; the case $\ell < 0$ is similar. 
If $h$ is the $\ell$-th cutting point of $(k,x)$, then $y_{[-k,h)} = \sigma(x_{[0,\ell)}) = \sigma(x'_{[\ell'-\ell,\ell')})$. 
Since $h$ is also the $\ell'$-th cutting point of $(k',x')$, the $(\ell'{-}\ell)$-th and $(\ell'{-}\ell+1)$-st $\sigma$-cutting points of $(k',x')$ are $-k$ and $|\sigma(x'_{\ell'-\ell})|-k = |\sigma(x_0)|-k$ respectively.
As $(k,x)$ is centered, we have $-k \le 0 < |\sigma(x_0)|-k$.
As $(k',x')$ is centered, we obtain that $\ell' - \ell = 0$ and $k' = k$, i.e., $(k,x) = (k',x')$. 
\end{proof}

\begin{lemma} \label{l:injectivecut}
Let $\sigma:\, \mathcal{A} \to \mathcal{B}^+$ be a morphism that is injective on~$\mathcal{A}^\mathbb{Z}$.
If $(k,x), (k',x')$ are distinct centered $\sigma$-representations of some $y \in \mathcal{B}^\mathbb{Z}$, then $(k,x)$ and $(k',x')$ have no common $\sigma$-cut. 
\end{lemma}

\begin{proof}
Let $(k,x), (k',x')$ be distinct centered $\sigma$-representations of some $y \in \mathcal{B}^\mathbb{Z}$.
Suppose that the $\ell$-th $\sigma$-cutting point of $(k,x)$ equals the $\ell'$-th $\sigma$-cutting point of $(k',x')$ for some $\ell, \ell' \in \mathbb{Z}$. 
Then we have $\sigma(T^\ell(x)) = \sigma(T^{\ell'}(x'))$ and thus $T^\ell(x) = T^{\ell'}(x')$ by the injectivity of~$\sigma$.
By Lemma~\ref{l:cuty}, this contradicts that $(k,x) \ne (k',x')$. 
\end{proof}

Note that two $\sigma$-representations with no common $\sigma$-cut always give rise to two distinct centered $\sigma$-representations (with no common $\sigma$-cut). 
Therefore, we omit the requirement of being centered if we have no common $\sigma$-cut.
The next lemma implies that the property of having no common cut propagates ``down'' in $S$-adic systems. 

\begin{lemma}\label{l:cutpropagation}
Let $\sigma:\, \mathcal{A} \to \mathcal{B}^+$ and $\tau:\, \mathcal{B}\to\mathcal{C}^+$ be morphisms, $z \in \mathcal{C}^\mathbb{Z}$. 
If $z$ has $\tau$-representations $(\ell,y)$ and $(\ell',y')$ with no common $\tau$-cut, and $(k,x), (k',x')$ are $\sigma$-representations of $y$ and~$y'$ respectively, then $(\ell+|\tau(y_{[-k,0)})|,x)$ and $(\ell'+|\tau(y'_{[-k',0)})|,x')$ are $\tau \sigma$-representations of~$z$ with no common $\tau \sigma$-cut.
\end{lemma}

\begin{proof}
Suppose that $(\ell+|\tau(y_{[-k,0)})|,x)$ and $(\ell'+|\tau(y'_{[-k',0)})|,x')$ have a common $\tau \sigma$-cut, which is the $m$-th cut of the former and the $m'$-th cut of the latter $\tau \sigma$-representation. 
We assume that both $m$ and $m'$ are positive; the other cases are treated similarly. 
Then 
\[
|\tau\sigma(x_{[0,m)})| - |\tau(y_{[-k,0)})| - \ell = |\tau\sigma(x'_{[0,m')})| - |\tau(y'_{[-k',0)})| - \ell',
\]
which implies that for $j = |\sigma(x_{[0,m)})|$, $j' = |\sigma(x'_{[0,m')})|$, we have
\[
|\tau(y_{[-k,j-k)}|  -  |\tau(y_{[-k,0)})| - \ell = |\tau(y'_{[-k',j'-k')})| - |\tau(y'_{[-k',0)})| - \ell',
\]
i.e., $|\tau(y_{[0,j-k)})|-\ell = |\tau(y'_{[0,j'-k')})|-\ell'$.
This exhibits a common $\tau$-cut of $(\ell,y)$ and $(\ell',y')$.
\end{proof}

The property of having $\sigma$-representations with no common $\sigma$-cut also propagates within sets of points having the same language. 

\begin{lemma} \label{l:oneforall}
Let $\sigma:\, \mathcal{A} \to \mathcal{B}^+$ be a morphism, $(X,T)$ a shift with $X \subseteq \mathcal{A}^\mathbb{Z}$, $y \in \mathcal{B}^\mathbb{Z}$.
If $y$ has two $\sigma$-representations in~$X$ with no common $\sigma$-cut, then each $\tilde{y} \in \mathcal{B}^\mathbb{Z}$ with $\mathcal{L}_{\tilde{y}} \subseteq \mathcal{L}_y$ has two $\sigma$-representations in~$X$ with no common $\sigma$-cut. 
\end{lemma}

\begin{proof}
Let $(k,x)$ and $(k',x')$ be $\sigma$-representations of~$y$ in~$X$ with no common $\sigma$-cut.
For $\tilde{y}$ with $\mathcal{L}_{\tilde{y}} \subseteq \mathcal{L}_y$, there is a sequence $(n_i)$ such that $\tilde{y} = \lim_{i\to\infty} T^{n_i}(y)$. 
Clearly, $(k{+}n_i,x)$ and $(k'{+}n_i,x')$ are $\sigma$-representations of~$T^{n_i}(y)$ in~$X$ with no common $\sigma$-cut; let $(j_i, \tilde{x}_i)$ and $(j_i', \tilde{x}_i')$ be the corresponding centered $\sigma$-representations. 
Passing to a subsequence if necessary, we can assume that the sequences $(j_i, \tilde{x}_i)$ and $(j'_i, \tilde{x}_i')$ converge to $(j,\tilde{x})$ and $(j',\tilde{x}')$ respectively. 
Then $(j,\tilde{x})$ and $(j',\tilde{x}')$ are $\sigma$-representations of~$\tilde{y}$ in~$X$.
As all pairs $(j_i, \tilde{x}_i), (j_i', \tilde{x}_i')$ have no common $\sigma$-cut, this also holds for $(j,\tilde{x}), (j',\tilde{x}')$.
\end{proof}

We can now prove Proposition~\ref{p:evrec2} under additional assumptions; the general case is proved in Section~\ref{sec:proof2} below.

\begin{proof}[Proof of Proposition~\ref{p:evrec2} for minimal shifts and injective substitutions]
Let $\boldsymbol{\sigma}$ be a directive sequence 
such that $|\mathcal{A}_n| \le K$, $(X_{\boldsymbol{\sigma}}^{(n)},T)$ is minimal, and $\sigma_n$ injective on $X_{\boldsymbol{\sigma}}^{(n+1)}$ for all $n\ge0$. 
Let $\ell = K + \lfloor \log_2(K-1) \rfloor $ and suppose that $\boldsymbol{\sigma}$ is not recognizable for aperiodic points at levels $n_1 < n_2 < \dots < n_\ell$.
Then, by Lemmas~\ref{l:injectivecut} and~\ref{l:oneforall}, each aperodic point $y \in X_{\boldsymbol{\sigma}}^{(n_i)}$, $1 \le i \le \ell$, has two $\sigma_{n_i}$-representations with no common $\sigma_{n_i}$-cut. 

Let $y \in X_{\boldsymbol{\sigma}}^{(n_1)}$ be an aperiodic point, and let $(k,x), (k',x')$ be $\sigma_{n_1}$-representations in~$X_{\boldsymbol{\sigma}}^{(n_1+1)}$ with no common $\sigma_{n_1}$-cut.
Both $x$ and~$x'$ have two $\sigma_{[n_1+1,n_2+1)}$-representations in~$X_{\boldsymbol{\sigma}}^{(n_2+1)}$ with no common $\sigma_{[n_1+1,n_2+1)}$-cut, thus $y$ has four $\sigma_{[n_1,n_2+1)}$-representations in~$X_{\boldsymbol{\sigma}}^{(n_2+1)}$ with pairwise no common $\sigma_{[n_1,n_2+1)}$-cut by Lemma~\ref{l:cutpropagation}.
Inductively, we get that $y$ has $2^\ell$ $\sigma_{[n_1,n_\ell+1)}$-representations in~$X_{\boldsymbol{\sigma}}^{(n_\ell+1)}$ with pairwise no common $\sigma_{[n_1,n_\ell+1)}$-cut.
As $2^\ell \ge 2^K (K-1) \ge 2^{K-1} (K-1)+2$, this contradicts Lemma~\ref{l:infection}.
\end{proof}

\subsection{Finding representations with no common cuts} \label{sec:finding-an-injective}
In this subsection, we show how to find points with two representations with no common cuts in sequences of morphisms that are non-recognizable at $K-1$ levels, where the alphabets are bounded by~$K$.
The following lemma, which tells us that, excluding the trivial case, composing enough morphisms on a bounded alphabet gives morphisms with a certain injectivity property relies on Lemma~\ref{l:reducealphabet} and is also inspired by~\cite{Down:2008}. 

\begin{lemma} \label{lem:common_cuts}
Let $K \ge 2$, let $\sigma_n:\, \mathcal{A}_{n+1} \to \mathcal{A}_n^+$, $1 \le n < K$, be morphisms with $|\mathcal{A}_K| \le K$.
If $|\sigma_{[1,K)}(\mathcal{A}_K^\mathbb{Z})| \ge 2$, then there exists $1 \le n < K$ such that $\sigma_n$ is injective on $\sigma_{[n+1,K)}(\mathcal{A}_K^\mathbb{Z})$. 
\end{lemma}

\begin{proof}
Suppose that $\sigma_n$ is not injective on $\sigma_{[n+1,K)}(\mathcal{A}_K^\mathbb{Z})$for each $1 \le n < K$. 
Then for each $1 \le n < K$ there exist $x, x' \in \mathcal{A}_K^\mathbb{Z}$ such that $\sigma_{[n,K)}(x) = \sigma_{[n,K)}(x')$ and $\sigma_{[n+1,K)}(x) \ne \sigma_{[n+1,K)}(x')$. 
It suffices to prove that this assumption implies that $|\sigma_{[1,K)}(\mathcal{A}_K^\mathbb{Z})|=1$.

Let first $n = K-1$, i.e., $\sigma_{K-1}(x) = \sigma_{K-1}(x')$, $x \ne x'$. 
By Lemma~\ref{l:reducealphabet}, we have an alphabet~$\mathcal{B}_{K-1}$ with $|\mathcal{B}_{K-1}| < |\mathcal{A}_K|$ and morphisms $\tau_{K-1}:\, \mathcal{A}_K \to \mathcal{B}_{K-1}^+$, $\tilde{\sigma}_{K-1}:\, \mathcal{B}_{K-1} \to \mathcal{A}_{K-1}^+$,~with \begin{equation}\label{eq:511st}
\sigma_{K-1} = \tilde{\sigma}_{K-1} \tau_{K-1}. 
\end{equation}
For $n = K-2$, by assumption we have $x, x' \in \mathcal{A}_K^\mathbb{Z}$ with
\begin{equation}\label{eq:511st2}
\sigma_{K-2} \sigma_{K-1} (x) = \sigma_{K-2} \sigma_{K-1} (x') \quad \mbox{and} \quad \sigma_{K-1}(x) \ne \sigma_{K-1}(x').
\end{equation}
In view of \eqref{eq:511st}, this implies that $\tau_{K-1}(x) \ne \tau_{K-1}(x')$. 
Thus, inserting \eqref{eq:511st} in \eqref{eq:511st2} we gain
\[
\sigma_{K-2} \tilde{\sigma}_{K-1} (\tau_{K-1}(x)) = \sigma_{K-2} \tilde{\sigma}_{K-1} (\tau_{K-1}(x')) \quad \mbox{and} \quad \tau_{K-1}(x) \ne \tau_{K-1}(x').
\]
Applying Lemma~\ref{l:reducealphabet} for $\sigma_{K-2} \tilde{\sigma}_{K-1}$ gives an alphabet $\mathcal{B}_{K-2}$ with $|\mathcal{B}_{K-2}| < |\mathcal{B}_{K-1}|$ and morphisms $\tau_{K-2}:\, \mathcal{B}_{K-1} \to \mathcal{B}_{K-2}^+$, $\tilde{\sigma}_{K-2}:\, \mathcal{B}_{K-2} \to \mathcal{A}_{K-2}^+$, such that $\sigma_{K-2} \tilde{\sigma}_{K-1} = \tilde{\sigma}_{K-2} \tau_{K-2}$, thus $\sigma_{[K-2,K)} = \tilde{\sigma}_{K-2} \tau_{K-2} \tau_{K-1}$.

Inductively, we obtain alphabets $\mathcal{B}_1, \dots, \mathcal{B}_{K-1}$ with $|\mathcal{B}_1| < |\mathcal{B}_2| < \cdots < |\mathcal{B}_{K-1}| < |\mathcal{A}_K| \le K$ and morphisms $\tilde{\sigma}_1:\, \mathcal{B}_1 \to \mathcal{A}_1^+$, $\tau_n:\, \mathcal{B}_{n+1} \to \mathcal{B}_n^+$, $1 \le n < K$, with $\mathcal{B}_K = \mathcal{A}_K$, such that $\sigma_{[1,K)} = \tilde{\sigma}_1 \tau_1 \tau_2 \cdots \tau_{K-1}$. 
This implies that $|\mathcal{B}_1| = 1$, i.e., $|\sigma_{[1,K)}(\mathcal{A}_K^\mathbb{Z})| = 1$.
 \end{proof}

The following variation of Lemma~\ref{l:injectivecut} states that injectivity of $\tau$ on the image of $\sigma$ leads to $\tau \sigma$-representations with no common cut.
Applying it to a non-recognizable morphism $\tau = \sigma_n$ and $\sigma = \sigma_{[n+1,K)}$ as in Lemma~\ref{lem:common_cuts}, we obtain $\sigma_{[n,K)}$-representations with no common cut.

\begin{lemma} \label{l:ikcuts}
Let $\sigma:\, \mathcal{A} \to \mathcal{B}^+$, $\tau:\, \mathcal{B} \to \mathcal{C}^+$ be morphisms, $(X,T)$ a shift with $X \subseteq \mathcal{A}^\mathbb{Z}$, $Y = \bigcup_{k\in\mathbb{Z}} T^k\sigma(X)$, and $z \in \mathcal{C}^\mathbb{Z}$. 
If $\tau$ is injective on $\sigma(X)$ and $z$ has two centered $\tau$-representations in~$Y$, then $z$ has two $\tau \sigma$-representations in~$X$ with no common $\tau \sigma$-cut. 
\end{lemma}

\begin{proof}
Let $(\ell,y), (\ell',y')$ be distinct centered $\tau$-representations of~$z$ in~$Y$.
As $Y = \bigcup_{k\in\mathbb{Z}} T^k\sigma(X)$, $y$ and $y'$ have centered $\sigma$-representations $(k,x)$ and $(k',x')$ in~$X$.
Set $h = \ell+|\tau(y_{[-k,0)})|$, $h' = \ell'+|\tau(y'_{[-k',0)})|$. 
Then $(h,x)$ and $(h',x')$ are centered $\tau \sigma$-representations of~$z$. 
Suppose that the $m$-th $\tau \sigma$-cutting point of $(h,x)$ equals the $m'$-th $\tau \sigma$-cutting point of $(h',x')$. 
This has the following consequences. 
Firstly, it implies that the
\begin{equation}\label{eq:513-equalcut} 
\mbox{$(|\sigma(x_{[0,m)})|{-}k)$-th $\tau$-cutting point of $(\ell,y)$} = \mbox{$(|\sigma(x'_{[0,m')})|{-}k')$-th $\tau$-cutting point of $(\ell',y')$}
\end{equation}
where we have assumed that $m, m' \ge 0$, the other cases being similar. 
Secondly, it yields $\tau\sigma(T^m(x)) = \tau\sigma(T^{m'}(x'))$, thus $\sigma(T^m(x)) = \sigma(T^{m'}(x'))$ by the injectivity of $\tau$ on $\sigma(X)$. 
Then 
\begin{equation}\label{eq:513-same}
T^{|\sigma(x_{[0,m)})|-k}(y) = \sigma(T^m(x)) = \sigma(T^{m'}(x')) =  T^{|\sigma(x'_{[0,m')})|-k'}(y').
\end{equation}
By \eqref{eq:513-equalcut} and \eqref{eq:513-same} we may apply Lemma~\ref{l:cuty} to obtain that $(\ell,y) = (\ell',y')$, a contradiction. 
\end{proof}

\subsection{Proof  of Proposition~\ref{p:evrec2}} \label{sec:proof2} 
Let $\boldsymbol{\sigma}$ be a directive sequence with $|\mathcal{A}_n| \le K$ and $|\{\mathcal{L}_x:\, x \in X_{\boldsymbol{\sigma}}^{(n)}\}| \le L$ for all $n\ge0$. 
Let $\ell = (K-1 + \lfloor \log_2 (K-1) \rfloor)  L$, and suppose that $\boldsymbol{\sigma}$ is not recognizable for aperiodic points at levels $n$ for all $0 \le n < (K-1) (\ell+1)$. (We pass from a directive sequence $\boldsymbol{\sigma}$ that is not recognizable for aperiodic points at $(K-1) (\ell+1)$ levels to such a sequence by telescoping and removing initial recognizable levels.) By Lemma~\ref{lem:common_cuts}, we have for all $0 \le i \le \ell$ some $n_i$ with  $i\, (K-1) \le n_i < (i+1) (K-1)$ such that $\sigma_{n_i}$ is injective on $\sigma_{[n_i+1,(i+1)(K-1))}(X_{\boldsymbol{\sigma}}^{(i+1)(K-1)})$. 
By Lemma~\ref{l:ikcuts}, there exists an aperiodic $x^{(i)} \in X_{\boldsymbol{\sigma}}^{(n_i)}$ with two  $\sigma_{[n_i,(i+1)(K-1))}$-representations having no common $\sigma_{[n_i,(i+1)(K-1))}$-cut. 
By Lemma~\ref{l:oneforall}, the same holds for all $x$ with $\mathcal{L}_x = \mathcal{L}_{x^{(i)}}$. 

As $\mathcal{L}_x = \mathcal{L}_y$ implies that $\mathcal{L}_{\sigma_n(x)} = \mathcal{L}_{\sigma_n(y)}$, we consider a set of rooted trees (a~rooted forest) defined in the following way. 
We have a vertex for each language $\mathcal{L}_y$, $y \in X_{\boldsymbol{\sigma}}^{(n_i)}$, $0 \le i \le \ell$, and, for $0 \le i < \ell$, the vertex $\mathcal{L}_y$ has all vertices $\mathcal{L}_x$ with $x \in X_{\boldsymbol{\sigma}}^{(n_{i+1})}$ and $\mathcal{L}_{\sigma_{[n_i,n_{i+1})}(x)} \subseteq \mathcal{L}_y$ as children. 
If $y$ is aperiodic and has two $\sigma_{[n_i,(i+1)(K-1))}$-representations with no common $\sigma_{[n_i,(i+1)(K-1))}$-cut, then we call this vertex a \emph {special vertex}.
Since there are at least $\ell+1 = (K-1+\lfloor \log_2 (K-1) \rfloor)  L+1$ special vertices and the forest has at most $L$ leaves, there exists a special vertex~$\mathcal{L}_y$, $y \in X_{\boldsymbol{\sigma}}^{(n_i)}$, such that each path from this vertex to a leaf contains at least $K + \lfloor \log_2(K-1) \rfloor$ special vertices. 
Then $y$ has $2^K (K-1)$ different $\sigma_{[n_i,(K-1)(\ell+1))}$-representations with no common $\sigma_{[n_i,(K-1)(\ell+1))}$-cut, contradicting Lemma~\ref{l:infection}.

This proves Proposition~\ref{p:evrec2} and also finishes the proof of Theorem~\ref{t:evrec}.

\subsection{Different languages of points in $S$-adic shifts}\label{sec:minimal}
In this subsection we establish sufficient conditions for the number of languages of points in~$X_{\boldsymbol{\sigma}}^{(n)}$ to be bounded in~$n$. 
First note that, for a shift~$(X,T)$, the number of different languages~$\mathcal{L}_x$, $x \in X$, is bounded below by the number of minimal components of~$X$.
However, there can be points of~$X$ that do not belong to any minimal component of~$X$.
For example, let $(X,T)$ be the shift generated by $\cdots 0011\cdots$, then the minimal components of $X$ are $\{\cdots 000\cdots\}$ and $\{\cdots 111\cdots\}$, and $\{\mathcal{L}_x:\, x \in X\} = \{\mathcal{L}_{\cdots 000\cdots}, \mathcal{L}_{\cdots 111\cdots}, \mathcal{L}_{\cdots 0011\cdots}\}$. 
For a shift $(X,T)$ and $y \in X$, the set $Y = \{x \in X:\, \mathcal{L}_x = \mathcal{L}_y\}$ is shift-invariant and the shift-orbit of each $x \in Y$ is dense in~$Y$, but $Y$ need not be closed, thus $Y$ need not be minimal. 

We use the following lemmas.

\begin{lemma} \label{l:Xlimitword}
Let $x \in \mathcal{A}_0^\mathbb{Z}$. 
If $x$ has a $\sigma_{[0,n)}$-representation for all $n\ge0$, then $x \in X_{\boldsymbol{\sigma}}$ or $x = T^k(y)$ for some limit word~$y$ of~$\boldsymbol{\sigma}$, $k \in \mathbb{Z}$.
\end{lemma}

\begin{proof}
Let $x \in \mathcal{A}_0^\mathbb{Z}$ and suppose that $x$ has a $\sigma_{[0,n)}$-representation $(k_n,x^{(n)})$ for each $n\ge0$.
If $x \notin X_{\boldsymbol{\sigma}}$, then there is an $\ell > 0$ such that $x_{[-\ell,\ell)}$ is not a subword of $\sigma_{[0,n)}(a)$ for all $n\ge0$, $a \in \mathcal{A}_n$. 
This means that each $(k_n,x^{(n)})$ has a $\sigma_{[0,n)}$-cutting point~$h_n$ with $|h_n| < \ell$. 
Let $h$ be such that $h_n = h$ for infinitely many~$n$. 
Then we also have $\sigma_{[0,n)}$-representations $(h,\tilde{x}^{(n)})$ of~$x$ for all $n\ge0$, hence $(0,\tilde{x}^{(n)})$ is a $\sigma_{[0,n)}$-representation of $T^{-h}(x)$ for all $n\ge0$, i.e., $T^{-h}(x)$ is a limit word of~$\boldsymbol{\sigma}$. 
\end{proof}

\begin{lemma} \label{l:limitwords}
Let $\boldsymbol{\sigma}$ be everywhere growing and $\liminf |\mathcal{A}_n| = K < \infty$. 
Then there are at most $K^2$ limit words of~$\boldsymbol{\sigma}$, and at most $K^2 - K$ of them are not in~$X_{\boldsymbol{\sigma}}$. 
\end{lemma}

\begin{proof}
Let $x$ be a limit word of~$\boldsymbol{\sigma}$, with $x = \sigma_{[0,n)}(x^{(n)})$ for all $n \ge 0$. 
By a Cantor diagonal argument, we can assume that $\sigma_n(x^{(n+1)}) = x^{(n)}$ for all $n \ge 0$. 
If $\boldsymbol{\sigma}$ is everywhere growing, then the sequence of two-letter words $x_{-1}^{(n)} x_0^{(n)}$  defines~$x$. 
As $x_{-1}^{(n)} x_0^{(n)}$ determines $x_{-1}^{(j)} x_0^{(j)}$ for $j<n$, there are at most~$K^2$ such sequences, and $x \ne \tilde{x}$ implies that the corresponding sequences $x_{-1}^{(n)} x_0^{(n)}$ and $\tilde{x}_{-1}^{(n)} \tilde{x}_0^{(n)}$ agree only for finitely many~$n$. 

If $x \notin X_{\boldsymbol{\sigma}}$, then we have $x_{-1}^{(n)} x_0^{(n)} \notin \mathcal{L}_{\boldsymbol{\sigma}}^{(n)}$ for some $n\ge0$ and thus $x_{-1}^{(N)} x_0^{(N)} \notin \mathcal{L}_{\boldsymbol{\sigma}}^{(N)}$ for all $N \ge n$. 
Choose $n \ge 0$ such that, for any two limit words $x, \tilde{x}$, we have $x_{-1}^{(n)} x_0^{(n)} \ne \tilde{x}_{-1}^{(n)} \tilde{x}_0^{(n)}$, and also, if $x \notin X_{\boldsymbol{\sigma}}$, then $x_{-1}^{(n)} x_0^{(n)} \notin \mathcal{L}_{\boldsymbol{\sigma}}^{(n)}$, and finally $|\mathcal{A}_n| = K$.
If there is a letter in $\mathcal{A}_n$ that does not occur as~$x_0^{(n)}$, then we have at most $K^2-K$ words. 
Otherwise, since $x_0^{(n)}x_1^{(n)}$ is a prefix of $\sigma_{[n,N)}(x_0^{(N)})$ for $N$ large enough we get $x_0^{(n)}x_1^{(n)}\in\mathcal{L}_{\boldsymbol{\sigma}}^{(n)}$. 
Thus for each $a \in \mathcal{A}_n$ there exists $b \in \mathcal{A}_n$ such that $ab \in \mathcal{L}_{\boldsymbol{\sigma}}^{(n)}$, and, 
hence, there can be at most $K^2-K$ words $ab\in\mathcal{A}^2$ with $ab \notin \mathcal{L}_{\boldsymbol{\sigma}}^{(n)}$.
\end{proof}

Note that the condition that $\boldsymbol{\sigma}$ is everywhere growing cannot be omitted in Lemma~\ref{l:limitwords}. 
For example, if $\sigma_n:\, \mathcal{A} \to \mathcal{A}$ is the identity substitution for all $n\ge0$, then all elements of~$\mathcal{A}^\mathbb{Z}$ are limit words of~$\boldsymbol{\sigma}$, and $X_{\boldsymbol{\sigma}} = \emptyset$ since $\mathcal{L}^{(0)}_{\boldsymbol{\sigma}}$ only contains $\mathcal{A}$ and the empty word.  

\begin{proposition} \label{p:minimalcomponents}
Let $\boldsymbol{\sigma}$ be an everywhere growing directive sequence with  $\liminf_{n\to\infty} |\mathcal{A}_n| = K < \infty$.
Then $|\{\mathcal{L}_x:\, x \in X_{\boldsymbol{\sigma}}\}| \le (K^2-3K+5)K/3$.
\end{proposition}

\begin{proof}
We prove the statement by induction on~$K$.  It holds for $K=1$; assume that it holds for $\liminf_{n\to\infty} |\mathcal{A}_n| < K$ and suppose that $\liminf_{n\to\infty} |\mathcal{A}_n| = K$.
For $k \ge 0$, we define recursively words $w^{(k)} \in \mathcal{L}_{\boldsymbol{\sigma}}$,  alphabets 
\[
\mathcal{A}_n^{(k)} = \{a \in \mathcal{A}_n:\, \mbox{$\sigma_{[0,n)}(a)$ does not contain $w^{(k)}$ as subword}\},
\]
and shifts $(Y_k,T)$ with 
\[
Y_k = \{x \in X_{\boldsymbol{\sigma}}:\, \mbox{$x$ has a $\sigma_{[0,n)}$-representation in $(\mathcal{A}_n^{(k)})^\mathbb{Z}$ for all $n\ge0$}\},
\]
in the following way.

Let $w^{(0)}$ be the empty word, i.e., $\mathcal{A}_n^{(0)} = \emptyset$ for all $n\ge0$ and $Y_0 = \emptyset$.
We stop at $k \ge 0$ if $|\{\mathcal{L}_x:\, x \in X_{\boldsymbol{\sigma}} \setminus Y_k\}| = 1$. 
If $|\{\mathcal{L}_x:\, x \in X_{\boldsymbol{\sigma}} \setminus Y_k\}| \ge 2$, then let $x, y \in X_{\boldsymbol{\sigma}} \setminus Y_k$, $w^{(k+1)} \in \mathcal{L}_y$, be such that $w^{(k+1)} \notin \mathcal{L}_x$.
Note that all points in $X_{\boldsymbol{\sigma}} \setminus Y_k$ contain $w^{(k)}$ because $\sigma_{[0,n)}$-representations of sequences not containing $w^{(k)}$ are in $(\mathcal{A}_n^{(k)})^\mathbb{Z}$.
Therefore, we can choose $w^{(k+1)}$ in a way that it contains~$w^{(k)}$, which implies that $\mathcal{A}_n^{(k)} \subseteq \mathcal{A}_n^{(k+1)}$. 

Since $x \in X_{\boldsymbol{\sigma}} \setminus Y_k$, the word~$x$ has no $\sigma_{[0,n)}$-representation in $(\mathcal{A}_n^{(k)})^\mathbb{Z}$ for some $n\ge0$ and thus, by the definition of $\mathcal{A}_n^{(k)}$, also for all sufficiently large~$n$. 
As $x$ has a $\sigma_{[0,n)}$-representation in~$\mathcal{A}_n^\mathbb{Z}$ by Lemma~\ref{at_least_one}, there exists $a_n \in \mathcal{A}_n \setminus \mathcal{A}_n^{(k)}$ such that $\sigma_{[0,n)}(a_n) \in \mathcal{L}_x$ and, hence, $\sigma_{[0,n)}(a_n)$ does not contain $w^{(k+1)}$ as a subword. 
This implies that $\mathcal{A}_n^{(k)} \cup\{a_n\}\subseteq \mathcal{A}_n^{(k+1)}$, thus
$|\mathcal{A}_n^{(k)}| < |\mathcal{A}_n^{(k+1)}|$ for all large~$n$. 
Moreover, as $w^{(k+1)} \in \mathcal{L}_{\boldsymbol{\sigma}}$, we have $|\mathcal{A}_n^{(k+1)}|< |\mathcal{A}_n|$ for all large~$n$. 

Since $|\mathcal{A}_n^{(0)}| < |\mathcal{A}_n^{(1)}| < \cdots < |\mathcal{A}_n^{(k)}| < |\mathcal{A}_n|$ for large~$n$, we have $|\{\mathcal{L}_x:\, x \in X_{\boldsymbol{\sigma}} \setminus Y_k\}| = 1$ for some $k < K$.
For this $k$, which we will fix from now, we have $\liminf_{n\to\infty} |\mathcal{A}_n^{(k)}| < K$. 
By Lemma~\ref{l:Xlimitword}, we have 
\[
Y_k \subseteq X_{\boldsymbol{\sigma}'} \cup \{T^\ell(x):\, \mbox{$x$ is a limit word of $\boldsymbol{\sigma}'$, $\ell \in \mathbb{Z}$}\}, 
\]
where $\boldsymbol{\sigma}' = (\sigma_n')_{n\ge0}$, with $\sigma_n': \mathcal{A}_{n+1}^{(k)} \to (\mathcal{A}_n^{(k)})^+$ being the restriction of $\sigma_n$ to~$\mathcal{A}_{n+1}^{(k)}$.
Hence
\begin{multline*}
|\{\mathcal{L}_x: x \in X_{\boldsymbol{\sigma}}\}| \le |\{\mathcal{L}_x: x \in Y_k\}| + |\{\mathcal{L}_x: x \in X_{\boldsymbol{\sigma}} \setminus Y_k\}| \\ \le |\{\mathcal{L}_x: x \in X_{\boldsymbol{\sigma}'}\}| + |\{x \in X_{\boldsymbol{\sigma}} \setminus X_{\boldsymbol{\sigma}'}: \mbox{$x$ limit word of $\boldsymbol{\sigma}'$}\}| + |\{\mathcal{L}_x: x \in X_{\boldsymbol{\sigma}} \setminus Y_k\}|.
\end{multline*}
By the induction hypothesis, we have $|\{\mathcal{L}_x:\, x \in X_{\boldsymbol{\sigma}'}\}| \le (K^2-5K+9)(K-1)/3$.
As $\boldsymbol{\sigma}$ is everywhere growing, the same holds for $\boldsymbol{\sigma}'$, and there are at most $(K-1)(K-2)$ limit words of~$\boldsymbol{\sigma}'$ that are not in~$X_{\boldsymbol{\sigma}'}$ by Lemma~\ref{l:limitwords}.
Putting everything together, we obtain that
\[
|\{\mathcal{L}_x: x \in X_{\boldsymbol{\sigma}}\}| \le \frac{(K^2-5K+9)(K-1)}{3} + (K-1)(K-2) + 1 = \frac{(K^2-3K+5)K}{3}. \qedhere
\] 
\end{proof}

The following proposition together with Theorem~\ref{t:evrec} proves Theorem~\ref{t:substrec}.

\begin{proposition} \label{p:sigmacomp}
Let $\sigma:\, \mathcal{A} \to \mathcal{A}^+$ be a substitution. 
Then $\{\mathcal{L}_x:\, x \in X_\sigma\}$ is a finite set.
\end{proposition}

\begin{proof}
First note that $X_\sigma = X_{\sigma^n}$ for all $n \ge 1$.
Therefore, we can assume that $\sigma^2(a)$ and $\sigma(a)$ share the same first letter and the same last letter for all $a \in \mathcal{A}$.
Let $\mathcal{A}_\mathrm{u}$ be the set of letters $a \in \mathcal{A}$ such that $|\sigma^n(a)| \to \infty$ and $\mathcal{A}_\mathrm{b} = \mathcal{A} \setminus \mathcal{A}_\mathrm{u}$.
Note that, if $a \in \mathcal{A}_b$ and $\sigma(a)$ contains~$a$, then $\sigma(a) = a$. 

Let $x$ be a limit word of $\boldsymbol{\sigma} = (\sigma)_{n\ge0}$.
If $x_0 \in \mathcal{A}_\mathrm{b}$, then we have $\sigma(x_0) = x_0$. 
Inductively, we obtain that $\sigma(x_j) = x_j$ for all $j \ge 0$ such that $x_{[0,j+1)} \in \mathcal{A}_\mathrm{b}^*$, as well as for all $j < 0$ such that $x_{[j,0)} \in \mathcal{A}_\mathrm{b}^*$. 
If $x_j \in \mathcal{A}_\mathrm{u}$ and $x_{[0,j)} \in \mathcal{A}_\mathrm{b}^*$, $j \ge 0$, then we obtain that $\sigma(x_{[j,\infty)}) = x_{[j,\infty)}$. 
Similarly, $x_j \in \mathcal{A}_\mathrm{u}$ and $x_{[j+1,0)} \in \mathcal{A}_\mathrm{b}^*$, $j < 0$, imply that $\sigma(x_{(-\infty,j]}) = x_{(-\infty,j]}$.
Hence $\sigma(x) = x$, and the limit words of~$\boldsymbol{\sigma}$ are exactly the fixed points of~$\sigma$. 

Now, we follow the proof of Proposition~\ref{p:minimalcomponents}. 
For $k \ge 0$, we define recursively words $w^{(k)} \in \mathcal{L}_\sigma$, alphabets $\mathcal{A}^{(k)}$ of letters $a \in \mathcal{A}$ such that $\sigma^n(a)$ does not contain $w^{(k)}$ as subword for all $n \ge 1$, and sets $Y_k$ of points $x \in X_\sigma$ having $\sigma^n$-representations in $(\mathcal{A}^{(k)})^\mathbb{Z}$ for all $n\ge0$: $w^{(0)}$ is the empty word; if $|\{\mathcal{L}_x:\, x \in X_\sigma \setminus Y_k\}| \ge 2$, then $w^{(k+1)}$ is chosen in a way that $w^{(k+1)} \in \mathcal{L}_y \setminus \mathcal{L}_x$ for some $x, y \in X_{\boldsymbol{\sigma}} \setminus Y_k$, and $w^{(k)}$ is a subword of~$w^{(k+1)}$. 
Then we have $|\mathcal{A}^{(0)}| < |\mathcal{A}^{(1)}| < \cdots < |\mathcal{A}^{(k)}| < |\mathcal{A}|$ and $|\{\mathcal{L}_x:\, x \in X_\sigma \setminus Y_k\}| = 1$ for some $k < K$.
For this~$k$, let $\sigma': \mathcal{A}^{(k)} \to (\mathcal{A}^{(k)})^+$ be the restriction of $\sigma$ to~$\mathcal{A}^{(k)}$.
As $Y_k \subseteq X_{\sigma'} \cup \{T^\ell(x):\, x \in X_\sigma \cap (\mathcal{A}^{(k)})^\mathbb{Z},\, \sigma'(x) = x,\, \ell \in \mathbb{Z}\}$, we have
\[
|\{\mathcal{L}_x: x \in X_\sigma\}| \le |\{\mathcal{L}_x: x \in X_{\sigma'}\}| + |\{\mathcal{L}_x:\, x \in X_\sigma,\, \sigma(x) = x\}| + 1.
\]
By induction on the size of the alphabet~$\mathcal{A}$, $\{\mathcal{L}_x: x \in X_{\sigma'}\}$ is a finite set, and it only remains to show that $\{\mathcal{L}_x:\, x \in X_\sigma,\, \sigma(x) = x\}$ is finite. 

Let $x \in X_\sigma$ with $\sigma(x) = x$. 
Suppose first that $x_j, x_k \in \mathcal{A}_\mathrm{u}$ and $x_{[j+1,k)} \in \mathcal{A}_\mathrm{b}^*$ for some $j < 0 \le k$.
We show that $k-j$ is bounded. 
Let $n \ge 0$ be minimal such that $x_{[j,k+1)}$ is a subword of $\sigma^{n+1}(a)$ for some $a \in \mathcal{A}$. 
Then there is a subword $\tilde{a} v a'$ of $\sigma(a)$ with $\tilde{a}, a' \in \mathcal{A}_\mathrm{u}$, $v \in \mathcal{A}_\mathrm{b}^*$, such that $x_{[j,k+1)}$ is a subword of $\sigma^n(\tilde{a} v a')$. 
Let $j < \tilde{m} \le m' \le k$ be such that $x_{[j,\tilde{m})}$ is a suffix of $\sigma^n(\tilde{a})$, $x_{[\tilde{m},m')} = \sigma^n(v)$ and $x_{[m',k+1)}$ is a prefix of $\sigma^n(a')$.
Since $v \in \mathcal{A}_\mathrm{b}^*$ and $|v| < |\sigma(a)|$, $m'-\tilde{m}$ is bounded.
As $x_{[m',k)} \in \mathcal{A}_\mathrm{b}^*$ and $\sigma(x_k)$ starts with~$x_k$, we obtain that $k-m'$ is bounded as well. 
Similarly, $\tilde{m}-j$ is bounded and, hence, $k-j$ is bounded. 
Recall that $\sigma(x_{[j+1,k)}) = x_{[j+1,k)}$. 
Since $x_{(-\infty,j]}$ and $x_{[k,\infty)}$ are determined by $x_j$ and~$x_k$, there are only finitely many possibilities for~$x$. 

Suppose now that there is some $k \ge 0$ such that $x_k \in \mathcal{A}_\mathrm{u}$ and $x_j \in \mathcal{A}_\mathrm{b}$ for all $j < k$. 
As $T^k(x)$ is also a fixed point of $\sigma$ in this case, we can assume that $k = 0$. 
For $j < 0$, let $n \ge 0$ minimal such that $x_{[j,1)}$ is a subword of $\sigma^{n+1}(a)$ for some $a \in \mathcal{A}$. 
Then there is a subword $v a'$ or $\tilde{a} v a'$ of $\sigma(a)$ with $\tilde{a}, a' \in \mathcal{A}_\mathrm{u}$, $v \in \mathcal{A}_\mathrm{b}^*$, such that $x_{[j,1)}$ is a subword of $\sigma^n(v a')$ or $\sigma^n(\tilde{a} v a')$. 
Let $j \le \tilde{m} \le m' \le 0$ be such that $x_{[j,\tilde{m})}$ is a suffix of $\sigma^n(\tilde{a})$, $x_{[\tilde{m},m')} = \sigma^n(v)$ and $x_{[m',1)}$ is a prefix of $\sigma^n(a')$, with $\tilde{m} = j$ when $x_{[j,1)}$ is a subword of $\sigma^n(v a')$,
As in the previous paragraph, $|m'|$ and $m' - \tilde{m}$ are bounded, thus $|\tilde{m}|$ is bounded. 
Therefore, there are some $\tilde{m} < 0$ and $\tilde{a} \in \mathcal{A}_\mathrm{u}$ such that, for infintely many $j<0$, $x_{[j,\tilde{m})}$ is a suffix of $\sigma^n(\tilde{a})$ for some $n \ge 0$.
Since this determines $x_{(-\infty,\tilde{m})}$, $x_0$ determines $x_{[0,\infty)}$ and $|\tilde{m}|$ is bounded, there are again only finitely many possibilities for~$x$. 

Finally, suppose that $x \in \mathcal{A}_\mathrm{b}^\mathbb{Z}$.
Now, for each $k > 0$, there exist $\tilde{a}, a' \in \mathcal{A}_\mathrm{u}$, $n \ge 0$, $-k \le \tilde{m} \le m' \le k$, such that $x_{[-k,\tilde{m})}$ is a suffix of $\sigma^n(\tilde{a})$, $x_{[m',k)}$ is a prefix of $\sigma^n(a')$, and $m'-\tilde{m}$ is bounded.
Moreover, if $k-m'$ is sufficiently large, then $\lim_{n\to\infty} \sigma^n(a'a'\cdots)$ is an eventually periodic sequence. 
Similarly, $\lim_{n\to\infty} \sigma^n(\cdots\tilde{a}\tilde{a})$ is eventually periodic (to the left) when $\tilde{m}+k$ is sufficiently large. 
Thus $x$ is periodic with period given by $\lim_{n\to\infty} \sigma^n(\cdots\tilde{a}\tilde{a})$ or $\lim_{n\to\infty} \sigma^n(a'a'\cdots)$, or we have some $\tilde{m} \le m'$, $\tilde{a}, a' \in \mathcal{A}_\mathrm{u}$, such that $x_{(-\infty,\tilde{m})} = \lim_{n\to\infty} \sigma^n(\cdots\tilde{a}\tilde{a})$, $x_{[m,\infty)} = \lim_{n\to\infty} \sigma^n(a'a'\cdots)$, and $m' - \tilde{m}$ is bounded.
Hence, up to shifting, there are only finitely many fixed points of this form. 
Therefore, $\{\mathcal{L}_x:\, x \in X_\sigma,\, \sigma(x) = x\}$ is a finite set, which proves the proposition.
\end{proof}

\section{Bratteli-Vershik representations} \label{sec:bratt-versh-repr}

In this section, we discuss the implications of recognizability for the natural representation of an $S$-adic shift as a Bratteli-Vershik system (defined below).  
A~Bratteli-Vershik system $(X_B,\varphi_\omega)$ can be topological, with the Vershik map~$\varphi_\omega$ a homeomorphism, or measurable, with $\varphi_\omega$ defined almost everywhere.  
A~measure theoretic Bratteli-Vershik representation has been established for primitive substitutive shifts by Livshits and Vershik in \cite{Livshits-Vershik} and Canterini and Siegel in \cite{CS01b, CanSie:2001}.
Durand, Host, and Skau \cite{Durand-Host-Skau}  found a topological representation for primitive substitutions that has been extended to aperiodic  (and not necessarily primitive) substitutions by Bezuglyi, Kwiatkowski, and Medynets \cite{Bezugly:2009}. 
More generally, measure theoretic Bratteli-Vershik representations always exist for ergodic systems \cite{Vershik:1985}, as do topological  Bratteli-Vershik representations for (not necessarily primitive) aperiodic Cantor  dynamical systems \cite{medynets:06}. 
The question is when a \emph{naturally constructed} representation exists. 
For example, the representations in \cite{Livshits-Vershik} and \cite{CS01b, CanSie:2001}, and the representations for proper substitutions in  \cite{Durand-Host-Skau} and \cite{Bezugly:2009} are defined very naturally using the given substitution. 
Our approach here is to forgo the topological representation, preferring instead a measurable representation using a natural Bratteli-Vershik representation. 
One result in this section is Theorem~\ref{thm:Bratteli-Vershik}, where we show that each recognizable $S$-adic shift defined by an everywhere growing directive sequence is in fact measurably conjugate to the determined Bratteli-Vershik system. 
Our result is irrespective of any finite invariant measure through which we consider our shift. 
The \emph{natural construction} of the representation of the $S$-adic shift $(X_{\boldsymbol{\sigma}},T)$ as a Bratteli-Vershik system relies on a  natural sequence of  generating partitions of~$X_{\boldsymbol{\sigma}}$  (which are Kakutani-Rohlin partitions) that is introduced in \eqref{eq:partisadic} and Lemma~\ref{l:generating}. 
It is in the spirit of the representation found by \cite{CS01b, CanSie:2001} in the substitutive setting.

\subsection{From $S$-adic shifts to Bratteli-Vershik maps }
\begin{definition}\label{Definition_Bratteli_Diagram}
A~\emph{Bratteli diagram} is an infinite graph $B=(V,E)$ such that the vertex set $V=\bigcup_{n\geq 0}V_n$ and the edge set $E = \bigcup_{n\geq 0} E_n$ are partitioned into pairwise disjoint, nonempty subsets $V_n$ and~$E_n$, where
\renewcommand{\theenumi}{\roman{enumi}}
\begin{enumerate}
\item
$V_0=\{v_0\}$ is a single point;
\item
$V_n$ and $E_n$ are finite sets;
\item
there exists a range map~$r:E\rightarrow V$ and a source map~$s:E\rightarrow V$ such that $r(E_n)= V_{n+1}$ and  $s(E_n)= V_n$ for each $n\geq 0$.
\end{enumerate}
\end{definition}

The set $V_n$ is called the $n$-th level of the diagram~$B$. 
A~finite or infinite sequence of edges~$(e_n)$ with $e_n\in E_n$ such that $r(e_n)=s(e_{n+1})$ is called a \emph{finite} or \emph{infinite path}, respectively. 
For $0 \le m<n$, $v \in V_{m}$ and $w \in V_{n}$, let $E(v,w)$ denote the set of all paths $e = (e_m,\ldots, e_{n-1})$ with $s(e_m)=v$ and $r(e_{n-1})=w$. 
For a Bratteli diagram~$B$, let $X_B$ be the set of infinite paths $(x_0,x_1,\dots)$ starting at the top vertex~$v_0$.
For a finite path $e =  (e_m,\ldots, e_{n-1})  \in E(v,w)$, let  $U(e) = \{x\in X_B:\, x_i=e_i,\, m\le i < n\}$.
We endow~$X_B$ with the topology generated by the cylinder sets~$U(e)$, $e\in E(v,w)$, $v\in V_m$, $w\in V_n$, $0\leq m<n$.  
With this topology, cylinder sets are clopen, and $X_B$ is a $0$-dimensional compact metric space that may contain isolated points.  

Given a Bratteli diagram~$B$, the $n$-th \emph{incidence matrix} $F_{n} = (f^{(n)}_{w,v})$, $n \ge 0$, is a $|V_{n+1}|\times |V_n|$ matrix whose entries $f^{(n)}_{w,v}$ are equal to the number of edges between the vertices $w\in V_{n+1}$ and $v\in V_{n}$, i.e., $f^{(n)}_{w,v} = |\{e\in E_n : s(e) = v, r(e) = w\}|$.
A~Bratteli diagram~$B$ is \emph{simple} if for each $n$ there exists $N>n$ such that $E(v,w) \ne \emptyset$ for all $v \in V_n$ and $w \in V_N$.
 
A~Bratteli diagram $B=(V,E) $ is called \emph{ordered} if a linear order ``$>$" is defined on every set $r^{-1}(v)$, $v\in \bigcup_{n\ge 1} V_n$. 
We use ~$\omega$ to denote the corresponding partial order on~$E$ and write $(B,\omega)$ when we consider $B$ with the ordering~$\omega$.
For example, Figure~\ref{Bratteli_pic} contains the first three levels of a Bratteli diagram with incidence matrices $F_0$ and~$F_1$, and where $r^{-1}(v)$ is linearly ordered.

\begin{figure}[ht]
\begin{tikzpicture}
 \draw (0,0) -- (0,-1);
  \draw (0,0) -- (-1,-1); \draw (0,0) -- (1,-1);
   \draw (-1,-1) -- (-2,-2); \draw (-1,-1) -- (-1,-2);  \draw (-1,-1) -- (0,-2);\draw (-1,-1) -- (1,-2);
    \draw (0,-1) -- (-1,-2);   \draw (0,-1) -- (0,-2);\draw (0,-1) -- (1,-2);\draw (0,-1) -- (2,-2);
    \draw (1,-1) -- (1,-2); \draw (1,-1) -- (2,-2);
    \draw [rounded corners] (-1,-1) -- (-0.9,-1.5) -- (-1,-2);
     \draw [rounded corners] (-1,-1) -- (-1.1,-1.5) -- (-1,-2);
     \draw [rounded corners] (1,-1) -- (1.6,-1.5) -- (2,-2);
     \draw [rounded corners] (1,-1) -- (1.4,-1.5) -- (2,-2);
     \draw [rounded corners] (1,-1) -- (1.7,-1.5) -- (2,-2);
     \draw[fill] (0,-1) circle [radius=0.05];  
      \draw[fill] (-1,-1) circle [radius=0.05]; 
       \draw[fill] (1,-1) circle [radius=0.05]; 
       \draw[fill] (-1,-2) circle [radius=0.05]; 
        \draw[fill] (0,-2) circle [radius=0.05]; 
         \draw[fill] (-1,-2) circle [radius=0.05]; 
          \draw[fill] (2,-2) circle [radius=0.05];
             \draw[fill] (-2,-2) circle [radius=0.05];   \draw[fill] (1,-2) circle [radius=0.05];
                \draw[fill] (0,0) circle [radius=0.05];
                \node at (0.3,0) {$\,\scriptstyle v_0$};
                \node at (1.1,-2) {$\,\scriptstyle v$};
                 \node at (0.5,-1.9) {$\,\scriptscriptstyle 0$};
                   \node at (0.4,-1.6) {$\,\scriptscriptstyle 1$};
                     \node at (0.85,-1.6) {$\,\scriptscriptstyle 2$};
                     \node at (4,-0.5){$\scriptstyle {F_0^T=\begin{pmatrix} 1& 1& 1
                     \end{pmatrix}}$};
                      \node at (-3,0){\tiny{Level 0}};
                       \node at (-3,-1){\tiny{Level 1}}; \node at (-3,-2){\tiny{Level 2}};
     \node at (4,-1.5){$ \scriptstyle{F_1^T=\begin{pmatrix} 1& 3& 1&1&0 \\ 0 & 1 &1 & 1& 1 \\0&0&0&1&4 
   \end{pmatrix}}$};
\end{tikzpicture} 
\caption{The first three levels of a Bratteli diagram with $r^{-1}(v)$ ordered.}
 \label{Bratteli_pic}
\end{figure}
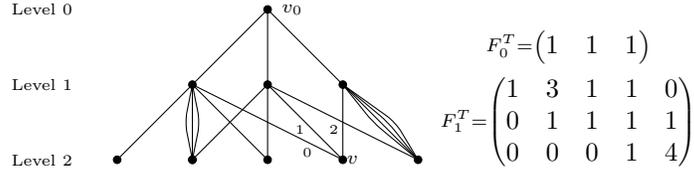

Let $m<n$ and let $(e_m,\dots,e_{n-1}) $ be a finite path in $B$. Its \emph{range} is $r(e_{n-1})$. An order~$\omega$ on~$B$ defines the \emph{lexicographic} ordering on the set of finite paths between vertices of levels $V_m$ and~$V_n$ with a common range: $(e_m,\dots,e_{n-1}) > (e'_m,\dots,e'_{n-1})$ if and only if there is $i$ with $m \le i < n$, $e_j=f_j$ for $i<j< n$ and $e_i> f_i$.
Thus, given an order~$\omega$, any two finite paths with a common range are comparable with respect to the lexicographic ordering generated by~$\omega$. 
If two infinite paths $(e_n)$ and  $(e'_n)$ are \emph{tail equivalent}, that is there exists $N$ such that  $e_n=e'_n$ for $n\geq N$,  then we can compare them by comparing their initial segments up to level~$N$. 
Thus $\omega$ defines a partial order on~$X_B$, where two infinite paths are comparable if and only if they are tail equivalent. 

We shall always assume that our Bratteli diagram~$B$ is  \emph{aperiodic} with respect to the tail equivalence relation, that is, no tail equivalence class is finite. 
Analogously to our earlier definition for directive sequences, we shall say that $B$ is \emph{everywhere growing} if $\lim_{n\to\infty} \min_{v\in V_n} |E(v_0, v)|= \infty$. 
Everywhere growing Bratteli diagrams are aperiodic.
 
Given an order~$\omega$ on a Bratteli diagram~$B$, we call a finite or infinite path $e=(e_n)$ \emph{maximal (minimal)} if every $e_n$ is maximal (minimal) amongst the edges from $r^{-1}(r(e_n))$.  
An ordered Bratteli diagram $(B, \omega)$ is called \emph{properly ordered} if the sets of minimal and maximal paths, $X_{\max}(\omega)$ and $X_{\min}(\omega)$, are singletons. 

Let $(B,\omega)$ be an ordered Bratteli diagram. 
If $x=(x_0, x_1, \ldots)$ is a non-maximal path in~$X_B$, then let $n$ be the smallest integer such that $x_n$ is not a maximal edge. 
Then the path $y=(y_0,y_1, \ldots)$, which agrees with $x$ from level $n{+}1$ onwards, where $(y_0,\ldots y_{n-1})$ is minimal and $y_n$ is the successor of~$x_n$, is called the \emph{successor} of~$x$ with respect to~$\omega$. 
We say that $\varphi_\omega:\, X_B  \to X_B$ is a \emph{Vershik map} if, for each $x \in X_B \setminus X_{\max}(\omega)$, $\varphi_\omega(x)$ is the successor of~$x$ with respect to~$\omega$.
We call $(X_B,\varphi_\omega)$ a \emph{Bratteli-Vershik dynamical system}.
Note that $\varphi_\omega:\, X_B \setminus X_{\max}(\omega) \to X_B \setminus X_{\min}(\omega)$ is bijective. 
In some cases, mainly when $\omega$ is a proper order, $\varphi_\omega$ can be extended to a homeomorphism on~$X_B$. 
Note that a properly ordered simple Bratteli diagram defines a minimal Cantor dynamical system.   
In this article, we do not assume that our orders are proper, nor are we concerned with whether $\varphi_\omega$ can be extended to a homeomorphism. 

For Bratteli diagrams whose vertex sets~$V_n$ are unbounded in size, the sets $X_{\min}(\omega)$ and~$X_{\max}(\omega)$ are not necessarily countable. 
For example, for a large family of diagrams defined in \cite{JQY}, almost any order defined on these diagrams will have uncountably many maximal and minimal paths. 
However, if $B$ is everywhere growing, we still have the following lemma, which is similar in spirit to \cite[Lemma~2.7]{BKMS:2010}; see also \cite[Prop 2.1]{FPS:2015} for a proof of the same result for the \emph{Pascal} Bratteli diagram.
Here, a~measure $\mu$ defined on the $\sigma$-algebra generated by cylinder sets in~$X_B$ is said to be \emph{invariant with respect to the tail equivalence relation} if, for any vertex $v\in V$, it gives equal $\mu$-mass to any cylinder set defined by a path from $v_0$ to~$v$.
  
\begin{lemma}\label{small_base}
Suppose that $B$ is an everywhere growing Bratteli diagram. 
Then for any order~$\omega$ on~$B$ and any probability measure~$\mu$ that is invariant with respect to the tail equivalence relation, we have $\mu(X_{\min}(\omega)) =\mu(X_{\max}(\omega)) = 0$. 
\end{lemma}

\begin{proof}
We use the fact that, for each~$n$, the set of paths from the top vertex~$v_0$ to level~$n$ of the Bratteli diagram defines cylinder sets which form a partition of~$X_B$.
Notice that the sets $X_{\max}(\omega)$ and $X_{\min}(\omega)$ are measurable as they are countable intersections of  finite unions of cylinder sets.
We show that  $\mu(X_{\max}(\omega))=0$, the other assertion following similarly. 
Fix an arbitrary~$n$. 
Since $B$ is everywhere growing, we can choose a level~$n$ so that any vertex in~$V_n$ has at least~$N$ incoming paths from~$v_0$. 
Let $C_n$ stand for the union of cylinder sets to level~$n$  which are defined by maximal paths. 
By the invariance of~$\mu$, we have $\mu(C_N) \le \frac{1}{N}$. 
Since $X_{\max}(\omega)$ is a subset of~$C_N$, we have $\mu(X_{\max}(\omega)) \le \frac{1}{N}$, and the result follows. 
\end{proof}

Note that if $\mu$ is a measure on~$X_B$ and $(X_B, \varphi_\omega,\mu)$ is a measure preserving system (in particular, this requires $\mu( X_{\max}(\omega)) = \mu(X_{\min}(\omega)) = 0$), then $\mu$ must be invariant with respect to the tail equivalence relation. 
Conversely, if $\mu$ on~$X_B$ is invariant under the tail equivalence relation and if a Vershik map~$\varphi_\omega$ on~$X_B$ is defined on a set of full $\mu$-mass, then $\mu$ is $\varphi_\omega$-invariant; see e.g.~\cite[Lemma~2.7]{BKMS:2010} or \cite[Proposition~2.11]{Fisher:09}.  
Finally, there is always at least one measure which is invariant with respect to the tail equivalence relation. 
Indeed, if the Bratteli diagram is simple, then it admits a proper order that makes the Vershik map a homeomorphism, so by the Krylov-Bogolyubov theorem, there is always at least one invariant measure w.r.t.\ the Vershik map and, hence, w.r.t.\ the tail equivalence relation. 
Otherwise, if $B$ is everywhere growing, then  the tail equivalence relation on $X_B$ is aperiodic.
Since $\bigcap_{n\geq k}F_k^T \ldots F_n^T \mathbb{R}_{+}^{|V_{n+1}|}$ is closed and non-empty for each~$k$, there exists a sequence of vectors $(p^{(n)})$ in $\mathbb R_{+}^{|V_n|}$ which satisfy $F_n^T p^{(n+1)}=p^{(n)}$ for each~$n$.
Now we can apply the second part of  \cite[Theorem 2.11]{BKMS:2010} to conclude that there exists at least one finite Borel measure which is invariant with respect to the tail equivalence relation.

We now give a certain way to represent elements of a recognizable $S$-adic shift that will be crucial for its relation with Bratteli-Vershik systems; see \cite{CS01b} for the case of substitutive shifts. 
Given a directive sequence~$\boldsymbol{\sigma}$ and $n\in \mathbb N$, define 
\begin{equation}\label{eq:partisadic}
\mathcal{P}_n = \{ T^k \sigma_{[0,n)}([a]):\, a \in \mathcal{A}_{n}, 0\leq k<|\sigma_{[0,n)}(a)| \}.
\end{equation}
We can think of $\mathcal{P}_n$ as consisting of $ |\mathcal A_n|$ towers. 
If ${\boldsymbol{\sigma}}$ is everywhere growing, then the minimum height of the towers in $\mathcal{P}_n$ tends to infinity.

For each $n\in \mathbb{N}$, $\mathcal{P}_n$~is a clopen cover of~$X_{\boldsymbol{\sigma}}$. We say that a sequence $(k_n, a_n)_{n\in\mathbb N}$ is a \emph{$(\mathcal{P}_n)$-address} for $x \in X_{\boldsymbol{\sigma}}$  if   $x\in T^{k_n} \sigma_{[0,n)}([a_n])$ for each $n\in \mathbb N$.    Each point $x \in X_{\boldsymbol{\sigma}}$ has at least one $(\mathcal{P}_n)$-address.
Let $\mu$ be  a shift invariant probability measure on~$X_{\boldsymbol{\sigma}}$.
The covers $(\mathcal{P}_n)_{n=0}^{\infty}$ are \emph{generating in  measure} if $\mu$-almost every $x \in X_{\boldsymbol{\sigma}}$ has a $(\mathcal{P}_n)$-address that uniquely determines~$x$. 
Limit words, i.e., points in $X_{\boldsymbol{\sigma}} \cap \bigcap_{n\in\mathbb{N}} \sigma_{[0,n)}(\mathcal{A}_n^\mathbb{Z})$, have $(\mathcal{P}_n)$-addresses $(0,a_n)_{n\in \mathbb N}$ that tell us nothing about their left infinite part. 
Thus if a one-sided limit point~$\bar{x}$ in $\bigcap_{n\in\mathbb{N}} \sigma_{[0,n)}(\mathcal{A}_n^\mathbb{N})$ can be preceded by two distinct ``pasts", i.e., if there exist two distinct limit points in $X_{\boldsymbol{\sigma}} \cap \bigcap_{n\in\mathbb{N}} \sigma_{[0,n)}(\mathcal{A}_n^\mathbb{Z})$ whose right infinite part is the same, then the $(\mathcal{P}_n)$-addresses with $k_n=0$ for all $n \ge 0$ do not distinguish these two points. 
On the other hand, we have the following lemma.

\begin{lemma}\label{l:generating}
Let $\boldsymbol{\sigma} = (\sigma_n)_{n\ge0}$ be a sequence of morphisms. 
If $x \in X_{\boldsymbol{\sigma}}$ does not belong to the shift orbit of a limit word of~$\boldsymbol{\sigma}$, then $x$ is uniquely determined by any $(\mathcal{P}_n)$-address of~$x$.
If $\boldsymbol{\sigma}$ is recognizable, then for all $n \ge 0$, $\sigma_n:\, X_{\boldsymbol{\sigma}}^{(n+1)} \to \sigma_n(X_{\boldsymbol{\sigma}}^{(n)})$ is a homeomorphism and the covers~$\mathcal{P}_n$ are partitions. 
If $\boldsymbol{\sigma}$ is recognizable and everywhere growing and $\mu$ is a shift invariant probability measure on~$X_{\boldsymbol{\sigma}}$, then $\mu(\bigcap_{n=0}^{\infty} \sigma_{[0,n)}(X_{\boldsymbol{\sigma}}^{(n)})) = 0$, and $(\mathcal P_n)$ is generating in measure.
\end{lemma}

\begin{proof}
Let $x \in X_{\boldsymbol{\sigma}}$ with $x \in T^{k_n}\sigma_{[0,n)}([a_n])$, $0 \le k_n < |\sigma_{[0,n)}(a_n)|$, for all $n \ge 0$. 
If $k_n$ is bounded, then we have some $k$ such that $k_n = k$ infinitely often, and $T^{-k}(x)$ is a limit word of~$\boldsymbol{\sigma}$. 
Similarly, if $|\sigma_{[0,n)}(a_n)|-k_n$ is bounded, then we have some $k$ such that $|\sigma_{[0,n)}(a_n)|-k_n = k$ infinitely often, and $T^k(x)$ is a limit word of~$\boldsymbol{\sigma}$. 
If both $k_n$ and $|\sigma_{[0,n)}(a_n)|-k_n$ are unbounded, then $x_{[-k_n, |\sigma_{[0,n)}(a_n)|-k_n)} = \sigma_{[0,n)}(a_n)$, $n \ge 0$, determines~$x$.

Recognizability immediately implies that $\sigma_n$ is a homeomorphism.  

If $\boldsymbol{\sigma}$ is recognizable and everywhere growing, then let $n \ge 0$ be such that $|\sigma_{[0,n)}(a)| \ge N$ for all $a \in \mathcal{A}_n$. 
As $\mathcal{P}_n$ forms a partition of~$X_{\boldsymbol{\sigma}}$ and $\mu$ is $T$-invariant, we have $\mu(\sigma_{[0,n)}(X_{\boldsymbol{\sigma}}^{(n)})) \le \frac{1}{N}$ and thus $\mu(\bigcap_{n=0}^{\infty} \sigma_{[0,n)}(X_{\boldsymbol{\sigma}}^{(n)})) = 0$.
The result follows.
\end{proof}

Given a directive sequence~$\boldsymbol{\sigma}$, we can define an associated \emph{natural}  ordered Bratteli diagram $(B,\omega)$, said to  be associated with~$\boldsymbol{\sigma}$, as follows. 
For $n\geq 1$, $V_n$~is a copy of~$\mathcal{A}_{n-1}$. 
The matrix $F_0$ is the vector $(1,1,\ldots,1)$ of size~$|\mathcal{A}_0|$. 
For $n\geq 1$, the $n$-th incidence matrix $F_n$ for the Bratteli diagram $B$ is the transpose of the incidence matrix of~$\sigma_{n-1}$. 
The substitutions~$\sigma_n$ also define the order $\omega$  that we consider on~$B$. 
Given a vertex $v \in V_{n+1}$ labelled by the letter $a \in \mathcal{A}_{n}$, we order the edges with range~$a$ as follows: if $b$ is the $j$-th letter in $\sigma_{n-1}(a)$, then we label an edge with source~$b$ and range~$a$ with $j{-}1$. 
In particular, if infinitely many of the morphisms~$\sigma_n$ are proper, then $\omega$ is a proper order, and $\varphi_\omega:\, X_B \to X_B$ is a homeomorphism.

Notice that $\boldsymbol{\sigma}$ is everywhere growing if and only if the diagram~$B$ in the natural Bratteli diagram $(B,\omega)$  associated with~$\boldsymbol{\sigma}$ is everywhere growing. 
Now we can reformulate Lemma~\ref{small_base} as follows: 

\begin{lemma}\label{small_base2}
Let $\boldsymbol{\sigma}$ be an  everywhere growing directive sequence and let  $(B,\omega)$ be the natural Bratteli  diagram associated  with $\boldsymbol{\sigma}$.
Then,   for any probability measure~$\mu$ that is invariant with respect to the tail equivalence relation, we have $\mu(X_{\min}(\omega)) =\mu(X_{\max}(\omega)) = 0$,  so
that  the Vershik map $\varphi_\omega$ is defined $\mu$-almost everywhere, and $ (X_B,\varphi_\omega, \mu)$   is a measure preserving  dynamical system.\end{lemma}

As, for everywhere growing~$\boldsymbol{\sigma}$, $(X_B,\varphi_\omega, \mu)$ is a measure preserving  dynamical system for \emph{any}~$\mu$ that is invariant with respect to the tail equivalence relation, we will abuse notation, and call $(X_B,\varphi_\omega)$ the \emph{natural (measurable) Bratteli-Vershik system} associated with~$\boldsymbol{\sigma}$.

We say that a transformation $\Phi:\, (X,T) \to (Y,S)$ is an \emph{almost-conjugacy} if there is a $T$-invariant set $\mathcal{D} \subset X$ and an $S$-invariant set $\mathcal{E} \subset Y$, with $\Phi:\, X \setminus \mathcal{D} \to Y \setminus \mathcal{E}$ a continuous bijection satisfying $\Phi \circ T= S \circ \Phi$, and such that~$\mathcal{D}$  (resp.~$\mathcal{E}$)     has zero measure for \emph{every} fully supported invariant probability measure on $(X,T)$ (resp.~$(Y,S)$).
Thus, if $(X,T)$ and $(Y,S)$ are almost conjugate, $\nu$~is any probability measure on~$X$ that is preserved by~$T$, and $\mu$~is any probability measure on~$Y$ is preserved by~$S$, then $(X,T,\nu)$ and $(Y,S,\mu)$ are measure theoretically conjugate.
If $\Phi$ is as above, except not injective, then we say that $(Y,S)$ is an \emph{almost-factor} of $(X,T)$.

\begin{theorem}\label{thm:Bratteli-Vershik}
Let $\boldsymbol{\sigma}$ be an everywhere growing directive sequence and $(X_B, \varphi_\omega)$ the natural Bratteli-Vershik dynamical system associated with~$\boldsymbol{\sigma}$. 
Then $(X_{\boldsymbol{\sigma}},T)$ is an almost-factor of $(X_B, \varphi_\omega)$. 
If $\boldsymbol{\sigma}$ is recognizable, then $(X_{\boldsymbol{\sigma}},T)$ is  almost-conjugate to $(X_B, \varphi_\omega)$.
\end{theorem}

\begin{remark}
If $\boldsymbol{\sigma}$ is everywhere growing and recognizable, then $(X_{\boldsymbol{\sigma}},T)$ is aperiodic. 
For, if $x\in X_{\boldsymbol{\sigma}}$ is periodic, choose $n$ such that for each letter $a\in A_n$,  $|\sigma_{[0,n)}(a)|$ is strictly larger than the period of~$x$. 
Then $x$ has at least two centered $\sigma_{[0,n)}$-representations.
\end{remark}

\begin{proof}
As $\boldsymbol{\sigma}$ is everywhere growing, $B$~is also.   
Let $\mathcal{E}$ denote the set of all paths in~$X_B$ that are tail equivalent to a minimal or maximal path. 
In other words, $\mathcal{E}$~is the union of the set of backward $\varphi_\omega$-orbits of maximal paths and the set of forward $\varphi_\omega$-orbits of minimal paths.  
By Lemma~\ref{small_base}, $\mathcal{E}$~has zero $\mu$-mass for any probability measure~$\mu$ on~$X_B$ that is invariant under the tail equivalence relation.

Define a map $\Psi:\, X_B \setminus \mathcal{E} \to X_{\boldsymbol{\sigma}}$ in the following way.
For a path $x = (x_0,x_1,\dots) \in X_B \setminus \mathcal{E}$ with, for $n \ge 1$, edges $x_n$ labelled by~$k_{n-1}$, vertices $r(x_n)$ labelled by $a_n \in \mathcal{A}_n$, words $p_{n-1} \in \mathcal{A}_{n-1}^{k_{n-1}}$ that are prefixes of $|\sigma_{n-1}(a_n)|$, we have that $p_{n-1} a_{n-1}$ is a prefix of $\sigma_{n-1}(a_n)$, hence
\[
q_n = \sigma_{[0,n-1)}(p_{n-1}) \sigma_{[0,n-2)}(p_{n-2}) \cdots \sigma_0(p_1) p_0
\]
is a prefix of $\sigma_{[0,n)}(a_n)$.
Let $y \in \mathcal{A}_0^\mathbb{Z}$ be such that $T^{-|q_n|}(y)$ starts with $\sigma_{[0,n)}(a_n)$ for all $n \ge 1$. 
As $k_{n-1} > 0$ for infinitely many $n \ge 1$ and $k_{n-1} < |\sigma_{n-1}(a_n)|$ for infinitely many $n \ge 1$, we have that $\lim_{n\to\infty} |q_n| = \infty$ and $\lim_{n\to\infty} |\sigma_{[0,n)}(a_n)| - |q_n| = \infty$.
Hence, $y$ is unique, and we set $\Psi(x) = y$.
Then $\Psi$ is continuous.
Let $\mathcal{D} = X_{\boldsymbol{\sigma}} \setminus \Psi(X_B \setminus \mathcal{E})$.

Let $\nu$ be a probability  measure on $X_{\sigma}$ which is  invariant under $T$. To prove that $(X_{\boldsymbol{\sigma}},T)$ is an almost-factor of $(X_B, \varphi_\omega)$, it remains to show that $\nu(\mathcal{D}) = 0$ for any $T$-invariant probability measure~$\nu$ on~$X_{\boldsymbol{\sigma}}$.
Similarly to Lemma~\ref{l:generating}, we obtain that each element of~$\mathcal{D}$ is in the shift orbit of a limit word of~$\boldsymbol{\sigma}$.
Similarly to Lemma~\ref{l:limitwords}, the property that $\boldsymbol{\sigma}$ is everywhere growing implies that the set of limit words of~$\boldsymbol{\sigma}$ is countable, thus $\mathcal{D}$ is countable.  
If $\nu(\{x\}) > 0$ for some $x \in X_{\boldsymbol{\sigma}}$, then $x$ is a periodic point.
As $\boldsymbol{\sigma}$ is everywhere growing, this implies that $x \in \Psi(X_B \setminus \mathcal{E})$.
Therefore, we have $\nu(\mathcal{D}) = 0$.

If $\boldsymbol{\sigma}$ is recognizable, it is straightforward to show that $\Psi$ is injective.
\end{proof}

\subsection{Recognizability and rank}

We  pause to discuss briefly the relation of recognizability to the multiple notions of \emph{rank}. 
The notion of rank has its roots in measurable transformations of the unit interval. 
One definition is that a transformation $T:\, [0,1]\rightarrow [0,1]$, which is measure preserving with respect to the Lebesgue measure~$\mu$, has \emph{(measurable) rank~$k$} if there is a sequence of measurable partitions $(\mathcal{P}_n)$ of $[0,1]$, where for each~$n$, $\mathcal{P}_n = \{T_1, \ldots ,T_k, S_n\}$,  with each $T_k$ a ``tower", with $S_n$ a  small ``spacer set", $\mu(S_n) \to 0$, and such that the sequence $(\mathcal{P}_n)$ generates the Borel $\sigma$-algebra. 
This comes naturally out of the ``cutting and stacking" definition of finite rank transformations. Here, if we interpret the notion of recognizability as the existence of a generating sequence of partitions, then recognizability is built in.
See \cite{Ferenczi:96,Fer97} for several possible definitions of rank~$k$ systems. 
Measurable Bratteli-Vershik systems often have a Lebesgue model as a cutting and stacking transformation, for example, most recently in \cite{AFP:2016}.
There is also a notion of topological rank. 
The \emph{rank} of a Bratteli diagram is the minimum~$k$ such that $|V_n|=k$ infinitely often. 
In \cite{Down:2008}, the authors define the (topological) \emph{rank} of a Cantor dynamical system $(X,S)$ to be the minimal~$k$ where $(X,S)$ is conjugate to a topological Bratteli-Vershik system $(X_B,\varphi_\omega)$, and where $B$ is of rank~$k$. Apart from the topological context, this notion of rank is close to that of measurable rank, for it also has a built in recognizability, in the form of a generating sequence of topological partitions of~$X_B$.

It also makes sense to talk about the rank of a recognizable $S$-adic shift. 
The \emph{rank} of a directive sequence~$\boldsymbol{\sigma}$ is the minimum~$k$ such that $|\mathcal A_n|=k$ infinitely often. Let  
$(X_{\boldsymbol{\sigma}}, T)$ be a recognizable, everywhere growing $S$-adic shift, with
$\boldsymbol{\sigma}$ of  rank $k$.
Then by Lemma~\ref{l:generating}, for any probability measure~$\nu$ such that $(X_{\boldsymbol{\sigma}}, T,\nu)$ is measure preserving, we have a sequence of partitions $(\mathcal{P}_n)$, each consisting of $k$ towers such that $(\mathcal{P}_n)$ generates the Borel $\sigma$-algebra defined by~$\nu$.

\begin{corollary}
Let $(X_{\boldsymbol{\sigma}}, T)$ be a recognizable, everywhere growing $S$-adic shift, with
$\boldsymbol{\sigma}$ of rank~$k$.  
Let $\nu$ be any probability measure such that $(X_{\boldsymbol{\sigma}}, T,\nu)$ is measure preserving. 
Then $(X_{\boldsymbol{\sigma}}, T,\nu)$ is measurably conjugate to a cutting and stacking measurable transformation of a Lebesgue space which is of measurable rank at most~$k$.
\end{corollary}

\begin{example}
The infimax $S$-adic family $\mathcal{F} = \{\theta_k:\, k\geq 1\}$, with $\theta_k:\, \{1,2,3\} \to \{1,2,3\}^+$, $\theta_k(1)=2$, $\theta_k(2)=31^{k+1}$, $\theta_k(3)=31^k$, is considered in \cite{Boyland:2016}, where Boyland and Severa consider $S$-adic shifts $(X_{\boldsymbol{\sigma} },T)$ defined by $\boldsymbol{\sigma} = (\sigma_n)_{n\geq 1}$ with each $\sigma_n\in \mathcal{F}$.
The authors show that there exist geometric realizations, via homeomorphisms, of such  $S$-adic shifts as interval translation maps. 
To do this, they modify the methods developed by Canterini and Siegel \cite{CS01b,CanSie:2001}, which involves using the ordered Bratteli diagram defined by~$\boldsymbol{\sigma}$, to  obtain a map from the $S$-adic shift to a generalized Rauzy fractal, such that the shift on~$X_{\boldsymbol{\sigma}}$ corresponds to an interval translation map for which the fractal is the attractor. 
It can easily be checked, either by hand or, since each~$M_{\theta_k}$ has rank~3, by applying Theorem~\ref{c:rec}, that any directive sequence chosen from~$\mathcal{F}$ yields a recognizable $S$-adic shift. 
Using Theorem~\ref{thm:Bratteli-Vershik}, we conclude that such a shift is almost conjugate to the natural (measurable) Bratteli-Vershik shift that it defines.
\end{example}

\begin{example}[Return words and tree shifts]
Let $X \subset \mathcal{A}^\mathbb{Z}$ be a minimal shift with language $\mathcal{L}_X$. 
Let $w \in \mathcal{L}_X$. 
A~(left) \emph{return  word} to~$w$ is a word $v \in \mathcal{L}_X$ that has $w$ as a prefix, such that $vw \in \mathcal{L}_X$ and such that $vw$ contains exactly two occurrences of~$w$. 
The set $\mathcal{R}_X(w)$ of return words to~$w$ is finite by minimality. 
A~\emph{coding morphism}  for $\mathcal{R}_X(w)$ is a morphism $\tau:\, \mathcal{B}^* \to  \mathcal{A}^*$ which maps the finite alphabet~$\mathcal{B}$ bijectively onto $\mathcal{R}_X(w)$. 

Let $\Gamma_X(w) = \{v \in  \mathcal{L}_X:\,  vw \in  \mathcal{L}_X \cap w \mathcal{A}^+\}$; then $\mathcal{R}_X(w) \subset  \Gamma_X(w)$. 
Elements $v$ of $\Gamma_X(w)$ have $w$ as a prefix but $vw$ might contain more than two occurrences of~$w$.
The shift with language $(\tau^{-1} \Gamma_X(w) ) \cup \{ \varepsilon\}$ is called the \emph{derived language} of~$X$ with respect to~$w$ and to~$\tau$ (here $\varepsilon$
stands for the empty word), and it is also a minimal shift. 
Let $x \in X$, let $\tau$ be a coding morphism with respect to the subword $x_{-1} x_0$, and 
let $y$ be the two-sided sequence defined by $S^{-1} x  = \tau(y)$. 
The point~$y$ belongs to~$Y$, the derived shift of~$X$ with respect to $w = x_{-1} x_0$ and~$\tau$.
We call the point~$y$ a \emph{derived point} of~$x$.
By iterating the process, one thus gets an $S$-adic representation of~$X$ obtained by return words, and expressed in terms of derived shifts; see e.g.\ \cite{Durand:00b} or  \cite{Durand-Leroy}. 
We describe the inductive procedure to construct this representation of $(X,T)$ as an S-adic shift. 
Fix $x=x^{(0)}\in X$, let $X^{(0)}=X$, let $\sigma_0=\tau$ be the coding morphism with respect to $x_{-1}^{(0)} x_0^{(0)}$, and let $X^{(1)}$ be the shift whose language is the derived language of~$X^{(0)}$ with respect to $x_{-1}^{(0)} x_0^{(0)}$. 
Now fix the $x^{(1)}\in X^{(1)}$ which satisfies $S^{-1} x^{(0)}  = \tau (x^{(1)})$, and let $\sigma_1$ be a coding morphism with respect to the subword $x_{-1}^{(1)}x_0^{(1)}$. 
Let $X^{(2)}$ be the shift whose language is the derived language of~$X^{(1)}$ with respect to $x_{-1}^{(1)}x_0^{(1)}$. 
Continuing inductively, we build a sequence of morphisms $\boldsymbol{\sigma} =(\sigma_n)_{n\geq 0}$ such that $\sigma_n$ is a coding morphism with respect to the subword $x_{-1}^{(n)}x_0^{(n)}\in X^{(n)}$,  and such that $X^{(n+1)}$ is the shift whose language is the derived language of~$X^{(n)}$ with respect to $x_{-1}^{(n)}x_0^{(n)}$ and~$\sigma_n$. 
Then $(X^{(n)})_{n\geq 0}=  (X_{\boldsymbol{\sigma}}^{(n)})_{n\geq 0}$ and $X= X_{\boldsymbol{\sigma}}^{(0)}$. 
Let us call this $S$-adic representation of~$X$ a \emph{return word $S$-adic representation of~$X$}.

A~priori, morphisms involved in such an $S$-adic representation are not substitutions: the  cardinality of $\mathcal{R}_X(w)$ is not necessarily equal to the cardinality of~$\mathcal{A}$.
We now consider a family of shifts~$X$, namely tree shifts, that have the striking property that  the sets of return words all have the same cardinality for every subword $v \in \mathcal{L}_X$ (they even generate the free group).
The family of tree shifts \cite{BerDeFel} includes Arnoux-Rauzy words and interval exchanges.
For a word $w \in \mathcal{L}_X$, we consider the undirected bipartite graph $G(w)$ called its \emph{extension graph} in $\mathcal{L}_X$ which is defined as follows.
Its set of vertices is the disjoint union of $\{a \in \mathcal{A}:\, aw \in  \mathcal{L}_X\}$ and $\{a \in \mathcal{A}:\, wa \in  \mathcal{L}_X\}$, and its edges are the pairs $(a,b)$ such that $awb \in \mathcal{L}_X$. 
We say that  the minimal shift~$X$ is a \emph{tree shift} if for every word~$w$, the graph $G(w)$ is a tree.

Tree shifts are stable under decoding by return words. 
Indeed, according to \cite[Theorem~5.13]{BerDeFel}, any derived shift of a minimal tree shift is a minimal tree shift defined on an alphabet of the same cardinality. 
Furthermore, the coding morphisms are free group automorphisms, hence their incidence matrices have full rank. 
In other words, by Theorem~\ref{c:rec}, if $(X_{\boldsymbol{\sigma}},T)$ is a return word $S$-adic representation of $(X,T)$, then it is a recognizable $S$-adic representation $(X_{\boldsymbol{\sigma}},T)=(X,T)$.  
The entire statement of Lemma~\ref{l:generating} holds. 
For, the fact that the morphisms in~$\boldsymbol{\sigma}$ are all defined on alphabets of the same cardinality implies that the set of limit words $\bigcap_{n=0}^{\infty} \sigma_{[0,n)}(X_{\boldsymbol{\sigma}}^{(n)})$ is finite. 
Let $(B,\omega)$ be the ordered Bratteli diagram associated with~$\boldsymbol{\sigma}$. Then $\mathcal{E}$, the set of all paths in~$X_B$ that are tail equivalent to a minimal or maximal path, is countable and so we can apply Theorem~\ref{thm:Bratteli-Vershik} to conclude that our tree shift $(X,T)$ is almost conjugate to the Bratteli-Vershik system $(X_B,\varphi_\omega)$. 

But in this case we can say more. 
Namely, we make a slight modification of the coding morphisms $(\sigma_n)_{n\geq 0}$ to render them proper. 
Sequences of proper morphisms are often useful in getting a topological Bratteli-Vershik representation, see e.g.\ \cite{Durand-Host-Skau,Durand-Leroy:2012}.
(However the fact that the sequence is proper is not a sufficient condition for a topological Bratteli-Vershik representation, cf.\ Example~\ref{ex:ws}.)
We describe the modification for a single coding morphism~$\tau$, which must then be applied to each~$\sigma_n$.
Consider the finite set $x_{-1}^{-1} \mathcal{R}_X(x_{-1} x_0)x_{-1}$.  
Any element~$v$ of this set has~$x_0$ as a prefix and $x_{-1}$ as a suffix. 
One checks that the coding morphism $\tau'$ for $ x_{-1} ^{-1} \mathcal{R}_X(x_{-1} x_0)x_{-1}$
is a proper substitution. 
It has also full rank since it has the same incidence matrix as~$\tau$. 
Hence the new sequence of coding morphisms $(\sigma_n)_{n\geq 0}$ is a sequence of proper morphisms, and the natural Bratteli-Vershik system associated with $\boldsymbol{\sigma}$ is properly ordered. 
\end{example}

To conclude, we have proved the following result that holds in particular for shifts  provided by  Arnoux-Rauzy  words or by codings of minimal interval exchanges.

\begin{theorem}\label{tree-topological-BV-rep}
Let $(X,T)$ be a minimal tree shift. 
Let $(X_{\boldsymbol{\sigma}},T)$ be a return word $S$-adic representation of $(X,T)$. 
Then $(X_B,\varphi_\omega)$, the natural Bratteli-Vershik system associated with~$\boldsymbol{\sigma}$, is properly ordered and is topologically conjugate to $(X,T)$. 
Its topological rank is bounded by the size of the alphabet of~$X$.
\end{theorem}

 \subsection{Consequences of eventual recognizability}
We end by returning to a question raised by Theorem~\ref{thm:Bratteli-Vershik}.
If $\boldsymbol{\sigma}$ is recognizable from level~$n$ on, and the number of centered $\sigma_{[0,n)}$-representations of $x \in X_{\boldsymbol{\sigma}}$ in $X_{\boldsymbol{\sigma}}^{(n)}$ is bounded, then the almost-factor map of Theorem \ref{thm:Bratteli-Vershik}
 is bounded-to-one. 
Indeed, if $\boldsymbol{\sigma}$ is recognizable from level~$n$ on, then, using the notation of Theorem \ref{thm:Bratteli-Vershik}, for any $x \in X_{\boldsymbol{\sigma}} \setminus \mathcal{E}$, $|\Psi^{-1}(\Psi(x))|$ is bounded by the number of centered $\sigma_{[0,n)}$-representations of $\Psi(x)$ in~$X_{\boldsymbol{\sigma}}^{(n)}$.
For an everywhere growing directive sequence~$\boldsymbol{\sigma}$ with alphabets of bounded size, we have recognizability for aperiodic points from some level~$n$ on by Theorem~\ref{t:evrec-1}.
If $X_{\boldsymbol{\sigma}}$ is aperiodic and the number of $\sigma_{[0,n)}$-representations of $x \in X_{\boldsymbol{\sigma}}$ in~$X_{\boldsymbol{\sigma}}^{(n)}$ is bounded, this would imply that $(X_{\boldsymbol{\sigma}},T)$ is a bounded-to-one almost-factor of $(X_B, \varphi_\omega)$. 

\begin{conjecture}
Let $\boldsymbol{\sigma}$ be an everywhere growing eventually recognizable directive sequence,  such that  the $S$-adic  system $(X_{\boldsymbol{\sigma}}, T)$ is an aperiodic shift. 
Let $(X_B, \varphi_\omega)$  be the natural Bratteli-Vershik dynamical system associated with~$\boldsymbol{\sigma}$. 
Is $(X_{\boldsymbol{\sigma}},T)$ is a bounded-to-one almost-factor of $(X_B, \varphi_\omega)$?
\end{conjecture} 

\section*{Acknowledgments}
The fourth author thanks IRIF, Universit\'e Paris Diderot--Paris 7, for its hospitality and support.

\bibliographystyle{amsalpha}
\bibliography{recog}
\end{document}